\documentclass[11pt]{amsart}
\usepackage{amsmath,amsthm,amssymb}
\usepackage{graphicx}
\usepackage[margin=1.2in]{geometry}

\begin{document}

%%%%%%%%%%%%%%%%%%%%%%%%%%%%%%%%%%%%%%%%%%%%%%%%%%%%%%%%%%%%%%%%%%%%%%%%%%%%%%%%%%%%%
%  Group Notation
%%%%%%%%%%%%%%%%%%%%%%%%%%%%%%%%%%%%%%%%%%%%%%%%%%%%%%%%%%%%%%%%%%%%%%%%%%%%%%%%%%%%%

\newcommand{\C} {{\mathbb C}}                              %complex
\newcommand{\R} {{\mathbb R}}                              %reals

\newcommand{\KG} {{\mathcal{K}} G}
\newcommand{\CG} {{\mathcal{C}} G}                %graph multiplihedra
\newcommand{\Y} {H}
\newcommand{\la} {\langle}
\newcommand{\ra} {\rangle}

\newcommand{\hide}[1]{}

\newcommand{\Sur} {\bold{S}}
\newcommand{\Mar} {\bold{M}}
\newcommand{\pop}   {\mathcal{P}}

\newcommand{\Mod}{\overline{\mathcal M}}     %compact moduli
\newcommand{\uMod}{{\mathcal M}}     %compact moduli
\newcommand{\oM} [1] {\ensuremath{{\mathcal M}_{0,#1}(\R)}}                 %open moduli
\newcommand{\M} [1] {\ensuremath{{\overline{\mathcal M}}{_{0, #1}(\R)}}}    %compact moduli
\newcommand{\cM} [1] {\ensuremath{{\mathcal M}_{0,#1}}}                     %open
\newcommand{\CM} [1] {\ensuremath{{\overline{\mathcal M}}{_{0, #1}}}}       %compact DMK

%%%%%%%%%%%%%%%%%%%%%%%%%%%%%%%%%%%%%%%%%%%%%%%%%%%%%%%%%%%%%%%%%%%%%%%%%%%%%%%%%%%%%
%  Labeling XXX
%%%%%%%%%%%%%%%%%%%%%%%%%%%%%%%%%%%%%%%%%%%%%%%%%%%%%%%%%%%%%%%%%%%%%%%%%%%%%%%%%%%%%

{\makeatletter
 \gdef\xxxmark{%
   % Addition to write list of xxx's to aux file.
   \protected@write\@auxout{\def\PAGE{ page }}
     {\@percentchar xxx: section \thesubsubsection \PAGE \thepage}%
   \expandafter\ifx\csname @mpargs\endcsname\relax % in minipage?
     \expandafter\ifx\csname @captype\endcsname\relax % in figure/caption?
       \marginpar{xxx}% not in a caption or minipage, can use marginpar
     \else
       xxx % notice trailing space
     \fi
   \else
     xxx % notice trailing space
   \fi}
 \gdef\xxx{\@ifnextchar[\xxx@lab\xxx@nolab}
 \long\gdef\xxx@lab[#1]#2{{\bf [\xxxmark #2 ---{\sc #1}]}}
 \long\gdef\xxx@nolab#1{{\bf [\xxxmark #1]}}
 % This turns them off:
 %\long\gdef\xxx@lab[#1]#2{}\long\gdef\xxx@nolab#1{}%
 \gdef\turnoffxxx{\long\gdef\xxx@lab[##1]##2{}\long\gdef\xxx@nolab##1{}}%
}

%%%%%%%%%%%%%%%%%%%%%%%%%%%%%%%%%%%%%%%%%%%%%%%%%%%%%%%%%%%%%%%%%%%%%%%%%%%%%%%%%%%%%%%%%%%
%
%  paper formatting
%
%%%%%%%%%%%%%%%%%%%%%%%%%%%%%%%%%%%%%%%%%%%%%%%%%%%%%%%%%%%%%%%%%%%%%%%%%%%%%%%%%%%%%%%%%%%

\theoremstyle{plain}
\newtheorem{thm}{Theorem}
\newtheorem{prop}[thm]{Proposition}
\newtheorem{cor}[thm]{Corollary}
\newtheorem{lem}[thm]{Lemma}
\newtheorem{conj}[thm]{Conjecture}

\theoremstyle{definition}
\newtheorem*{defn}{Definition}
\newtheorem*{exmp}{Example}

\theoremstyle{remark}
\newtheorem*{rem}{Remark}
\newtheorem*{hnote}{Historical Note}
\newtheorem*{nota}{Notation}
\newtheorem*{ack}{Acknowledgments}
\numberwithin{equation}{section}

%%%%%%%%%%%%%%%%%%%%%%%%%%%%%%%%%%%%%%%%%%%%%%%%%%%%%%%%%%%%%%%%%%%%%%%%%%%%%%%%%%%%%%%%%%%

\title{Deformations of bordered Riemann surfaces and associahedral polytopes}


\author{Satyan L.\ Devadoss}
\address{S.\ Devadoss: Williams College, Williamstown, MA 01267}
\email{satyan.devadoss@williams.edu}

\author{Timothy Heath}
\address{T.\ Heath: Columbia University, New York, NY 10027}
\email{timheath@math.columbia.edu}

\author{Cid Vipismakul}
\address{C.\ Vipismakul: University of Texas, Austin, TX 78712}
\email{wvipismakul@math.utexas.edu}

\begin{abstract}
We consider the moduli space of bordered Riemann surfaces with boundary and marked points.  Such spaces appear in open-closed string theory, particularly with respect to holomorphic curves with Lagrangian submanifolds.  We consider a combinatorial framework to view the compactification of this space based on the pair-of-pants decomposition of the surface, relating it to the well-known phenomenon of bubbling.  Our main result classifies all such spaces that can be realized as convex polytopes.  A new polytope is introduced based on truncations of cubes, and its combinatorial and algebraic structures are related to generalizations of associahedra and multiplihedra.
\end{abstract}


\maketitle

\baselineskip=17pt

% Things still to do:     TILING by cubeahedra

%%%%%%%%%%%%%%%%%%%%%%%%%%%%%%%%%%%%%%%%%%%%%%%%%%%%%%%%%%%%%%%%%%%%%%%%%%%%%%%%%%%%%
%
%                Overview
%
%%%%%%%%%%%%%%%%%%%%%%%%%%%%%%%%%%%%%%%%%%%%%%%%%%%%%%%%%%%%%%%%%%%%%%%%%%%%%%%%%%%%%
\section{Overview}

The moduli space $\uMod_{g,n}$ of Riemann surfaces of genus $g$ with $n$ marked particles has become a central object in many areas of mathematics and theoretical physics, ranging from operads, to quantum cohomology, to symplectic geometry.  This space has a natural extension by considering Riemann surfaces with \emph{boundary}.
Such objects typically appear in open-closed string theory, particularly with respect to holomorphic curves with Lagrangian submanifolds.
Although there has been much study of these moduli spaces from an analytic and geometric setting, a combinatorial and topological viewpoint has been lacking.  
At a high level, this paper can be viewed as an extension of the work by Liu \cite{liu} on the moduli of $J$-holomorphic curves and open Gromov-Witten invariants.  In particular, Liu gives several examples of moduli spaces of bordered Riemann surfaces, exploring their stratifications.  We provide a combinatorial understanding of this stratification and classify all such spaces that can be realized as convex polytopes, relating it to the well-known phenomenon of bubbling. 

There are several (overlapping) fields which touch upon ideas in this paper.  
%\begin{enumerate}
%\item 
Similar to our focus on studying bordered surfaces with marked points, recent work spearheaded by Fomin, Shapiro and Thurston \cite{fst} has established a world of cluster algebras related to such surfaces.   This brings in their notions of triangulated surfaces and tagged arc complexes, whereas our viewpoint focuses on the pair of pants decomposition of these surfaces.  
%\item 
Another interest in these ideas come from algebraic and symplectic geometery.  Here, an analytic approach can be taken to construct moduli spaces of bordered surfaces and their relationship to Gromov-Witten theory, such as the one by Fukaya and others  \cite{fuk}, or an algebraic method can used, like that of Harrelson, Voronov and Zuniga, to formulate a Batalin-Vilkovisky quantum master equation for open-closed string theory \cite{hvz}.
%\item 
A third field of intersection comes from the world of operads and algebraic structures.  Indeed, the natural convex polytopes appearing in our setting fit comfortably in the framework of higher category theory and the study of $A_\infty$ and $L_\infty$ structures seen from generalizations of associahedra, cyclohedra and multiplihedra \cite{df}.
%\end{enumerate}

% For instance, similar to our focus on studying bordered surfaces with marked points, recent work spearheaded by Fomin, Shapiro and Thurston \cite{fst} has established a world of cluster algebras related to such surfaces.   This brings in their notions of triangulated surfaces and tagged arc complexes, whereas our viewpoint focuses on the pair of pants decomposition of these surfaces.   Another interest in these ideas come from algebraic and symplectic geometery.  Here, either  an analytic approach is taken to construct moduli spaces of bordered surfaces and their relationship to Gromov-Witten theory  \cite{fuk}, or an algebraic one is used to set up a Batalin-Vilkovisky Quantum Master Equation for open-closed string theory \cite{hvz}.  A third field of intersection comes from the world of operads and algebraic structures.  Indeed, the natural convex polytopes which appear in our setting fit comfortably in the framework of higher category theory and the study of $A_\infty$ and $L_\infty$ structures \cite{hvz} seen from generalizations of associahedra, cyclohedra and multiplihedra \cite{df}.

An overview of the paper is as follows:  Section~\ref{s:defns} supplies a review of the definitions of the moduli spaces of interest.  Section~\ref{s:exmps} constructs the moduli spaces and provides detailed examples of several low-dimensional cases and their stratifications.  The main theorems classifying the polytopal spaces are provided in Section~\ref{s:convex} as well as an introduction to the associahedron and cyclohedron polytopes.  A new polytope is introduced and constructed in Section~\ref{s:cubes} based on the moduli space of the annulus.  Finally, the combinatorial and algebraic properties of this polytope are explored and related to the multiplihedron in Section~\ref{s:combin}.

%%%%%%%%%%%%%%%%%%%%%%%%%%%%%%%%%%%%%%%%%%%%%%%%%%%%%%%%%%%%%%%%%%%%%%%%%%%%%%%%%%%%%

\begin{ack}
We thank Chiu-Chu Melissa Liu for her enthusiasm, kindness, and patience in explaining her work, along with Ben Fehrman, Stefan Forcey, Jim Stasheff and  Aditi Vashista for helpful conversations.  We are also grateful to the NSF for partially supporting this work with grant DMS-0353634.  The first author also thanks MSRI for their hospitality, support and stimulating atmosphere in Fall 2009 during the Tropical Geometry and Symplectic Topology programs.
\end{ack}

%%%%%%%%%%%%%%%%%%%%%%%%%%%%%%%%%%%%%%%%%%%%%%%%%%%%%%%%%%%%%%%%%%%%%%%%%%%%%%%%%%%%%
%
%                Definitions
%
%%%%%%%%%%%%%%%%%%%%%%%%%%%%%%%%%%%%%%%%%%%%%%%%%%%%%%%%%%%%%%%%%%%%%%%%%%%%%%%%%%%%%
\section{Definitions} \label{s:defns}
\subsection{}

We introduce notation and state several definitions as we look at moduli spaces of surfaces with boundary.  Although we cover these ideas rather quickly, most of the detailed constructions behind our statements can be found in Abikoff \cite{ab}, Sep\"{a}l\"{a} \cite[Section 3]{sep}, Katz and Liu \cite[Section 3]{kl} and Liu \cite[Section 4]{liu}.

A smooth connected oriented \emph{bordered Riemann surface} $\Sur$ of type $(g,h)$ has genus $g \geq 0$ with $h \geq 0$ disjoint ordered circles $B_1, \ldots, B_h$ for its boundary.  
We assume the surface is compact whose boundary is equipped with the holomorphic structure induced by a holomorphic atlas on the surface.  Specifically, the boundary circles will always be given the orientation induced by the complex structure.
The surface has a \emph{marking set} $\Mar$ of type $(n, \bold m)$ if there are $n$ labeled marked points in the interior of $\Sur$ (called \emph{punctures}) and $m_i$ labeled marked points on the boundary component $B_i$, where $\bold m = \la m_1, \ldots, m_h \ra$.  
Throughout the paper, we define $m := m_1 + \cdots + m_h$.

\begin{defn}
The set $(\Sur, \Mar)$ fulfilling the above requirements is called a \emph{marked bordered Riemann surface}.  We say $(\Sur, \Mar)$ is \emph{stable} if its automorphism group is finite.
\end{defn}

Figure~\ref{f:bordered}(a) shows an example of $(\Sur, \Mar)$ where $\Sur$ is of type $(1,3)$ and $\Mar$ is of type $(3, \la 1,2,0 \ra)$.
%In particular, $(\Sur, \Mar)$ is stable if and only if $2g + h + 2n + m \ > \ 2 \; .$
Indeed, any stable marked bordered Riemann surface has a unique hyperboic metric such that it is compatible with the complex structure, where all the boundary circles are geodesics, all punctures are cusps, and all boundary marked points are half cusps.
We assume all our spaces $(\Sur, \Mar)$ are stable throughout this paper.

\begin{figure}[h]
\includegraphics{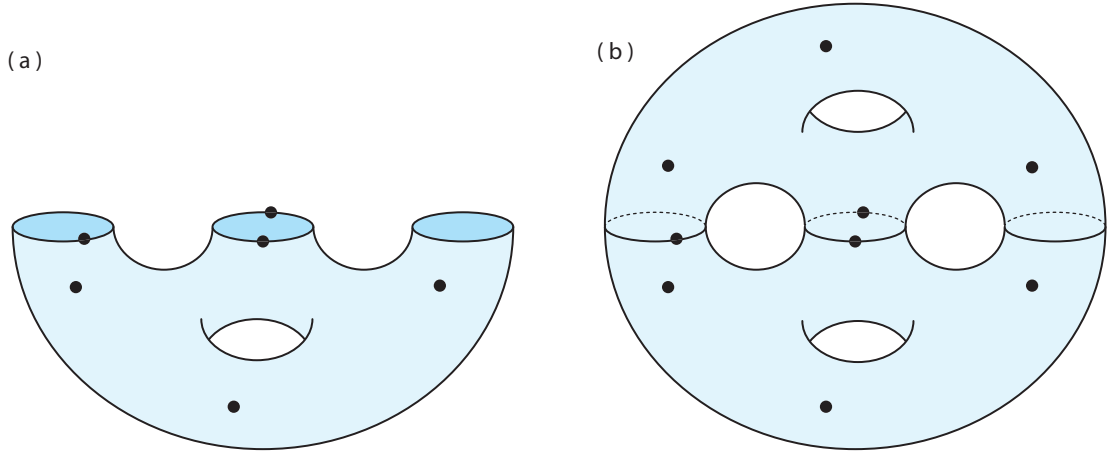}
\caption{(a) An example of a marked bordered Riemann surface (b) along with its complex double.}
\label{f:bordered}
\end{figure}

%Due to stability, we assume $(\Sur, \Mar)$ is none of the following throughout this paper:
%\begin{itemize}
%\item a sphere with less than three punctures;
%\item a disk with less than two marked points on the boundary;
%\item a sphere with less than three boundary components.
%\end{itemize}

%%%%%%%%%%%%%%%%%%%%%%%%%%%%%%%%%%%%%%%%%%%%%%%%%%%%%%%%%%%%%%%%%%%%%%%%%%%%%%%%%%%%%
\subsection{}

The \emph{complex double} $\Sur_\C$ of a bordered Riemann surface $\Sur$ is the oriented double cover of $\Sur$ without boundary.\footnote{
	The complex double and the \emph{Schottky double} of a Riemann surface coincide since it is orientable.
}
It is formed by gluing $\Sur$ and its mirror image along their boundaries; see \cite{ag} for a detailed construction.  For example, the disk is the surface of type $(0,1)$ whose complex double is a sphere, whereas the annulus is the surface of type $(0,2)$ whose complex double is a torus.  
Figure~\ref{f:bordered}(b) shows the complex double of $(\Sur, \Mar)$ from part (a).
In the case when $\Sur$ has no boundary, the double $\Sur_\C$ is simply the trivial disconnected double-cover of $\Sur$.  

The pair $(\Sur_\C, \sigma)$ is called a \emph{symmetric Riemann surface}, where $\sigma: \Sur_\C \to \Sur_\C$ is the anti-holormorphic involution.  The symmetric Riemann surface with a marking set $\Mar$ of type $(n, \bold m)$ has an involution $\sigma$ together with $2n$ distinct interior points $\{p_1, \ldots, p_n, q_1, \ldots, q_n\}$ such that $\sigma(p_i) = q_i$, along with $m$ boundary points $\{b_1, \ldots, b_m\}$ such that $\sigma(b_i) = b_i$.

\begin{defn}
A \emph{pair of pants} is a sphere from which three points or disjoint closed disks have been removed.  A pair of pants can be equipped with a unique hyperbolic structure compatible with the complex structure such that the boundary curves are geodesics.
\end{defn}

Let $\Sur$ be a surface without boundary.  A \emph{decomposition} of $(\Sur, \Mar)$ into pairs of pants is a collection of disjoint pairs of pant on $\Sur$ such that their union covers the entire surface and their pairwise intersection (of their closures) are either empty or a union of marked points and closed geodesic curves on $\Sur$.  Indeed, all the marked points of $\Mar$ appear as boundary components of pairs of pants in any decomposition of $\Sur$.  
%A disjoint set of  curves that decompose the surface into pairs of pants will be realized by a set of disjoint geodesics, and since there is only one geodesic in each homotopy  class, it will be a unique realization.
A disjoint set of  curves decomposing the surface into pairs of pants will be realized by a unique set of disjoint geodesics since there is only one geodesic in each homotopy class.

We now extend this decomposition to include marked surfaces with boundary.
Consider the complex double $(\Sur_\C, \sigma)$ of a marked bordered Riemann surface $(\Sur, \Mar)$.  If $\pop$ is a decomposition of $\Sur_\C$ into pairs of pants, then $\sigma(\pop)$ is another decomposition into pairs of pants.   The following is a generalization of the work of Sepp\"{a}l\"{a}:

\begin{lem} \cite[Section 4]{liu}
There exists a decomposition of $\Sur_\C$ into pairs of pants $\pop$ such that $\sigma(\pop) = \pop$ and the decomposing curves are simple closed geodesics of $\Sur_\C$.
\end{lem}

\noindent
Figure~\ref{f:bordered-pants}(a) shows examples of some of the geodesic arcs from a decomposition of $\Sur_\C$, where part (b) shows the corresponding decomposition for $\Sur$.  
Notice that there are three types of decomposing geodesics $\gamma$.
\begin{enumerate}
\item Involution $\sigma$ fixes all points on $\gamma$:  The geodesic must be a boundary curve of $\Sur$, such as the curve labeled $x$ in Figure~\ref{f:bordered-pants}(b).
\item Involution $\sigma$ fixes no points on $\gamma$:  The geodesic must be a closed curve on $\Sur$, such as the curve labeled $y$ in Figure~\ref{f:bordered-pants}(b).  
\item Involution $\sigma$ fixes two points on $\gamma$:  The geodesic must be an arc on $\Sur$, with its endpoints on the boundary of the surface, such as the curve labeled $z$ in Figure~\ref{f:bordered-pants}(b).
\end{enumerate}

\begin{figure}[h]
\includegraphics{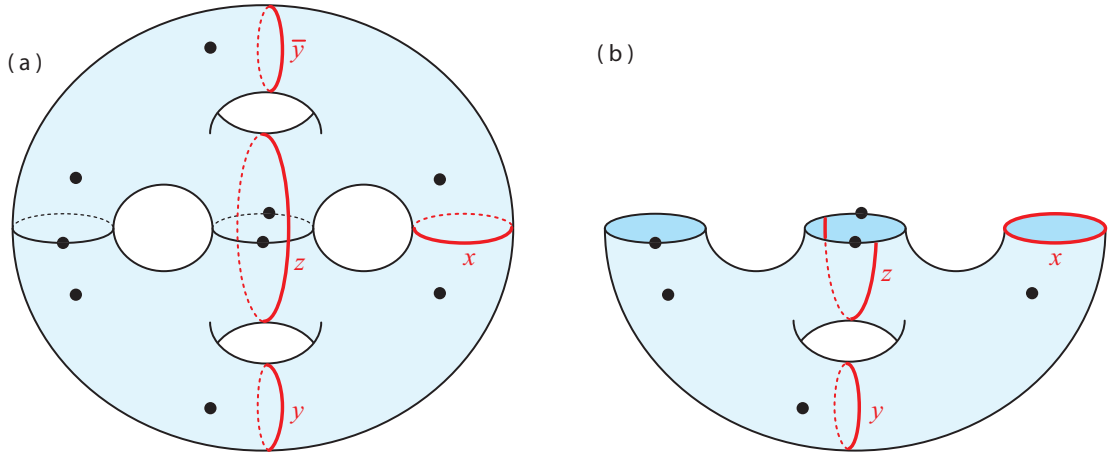}
\caption{Examples of some geodesic arcs from a pair of pants decomposition of (a) the complex double and (b) its marked bordered Riemann surface.}
\label{f:bordered-pants}
\end{figure}

We assign a \emph{weight} to each type of decomposing geodesic in a pair of pants decomposition, corresponding to the number of \emph{Fenchel-Nielsen coordinates} needed to describe the geodesic.  The geodesic of type (2) above has weight two because it needs two Fenchel-Nielsen coordinates to describe it (length and twisting angle), whereas geodesics of types (1) and (3) have weight one (needing only their length coordinates). 
The following is obtained from a result of Abikoff \cite[Chapter 2]{ab}.

\begin{lem} \label{l:weight}
Every pair of pants decomposition of a marked bordered Riemann surface $(\Sur, \Mar)$ of type $(g,h)$ with marking set $(n, \bold m)$ has a total weight of
\begin{equation} \label{e:dim}
6g + 3h - 6 + 2n + m
\end{equation}
from the weighted decomposing curves.
\end{lem}

%%%%%%%%%%%%%%%%%%%%%%%%%%%%%%%%%%%%%%%%%%%%%%%%%%%%%%%%%%%%%%%%%%%%%%%%%%%%%%%%%%%%%
\subsection{}

Based on the discussion above, we can reformulate the decomposing curves on marked bordered Riemann surfaces in a combinatorial setting:

\begin{defn}
An \emph{arc} is a curve on $\Sur$ such that its endpoints are on the boundary of $\Sur$, it does not intersect $\Mar$ nor itself, and it cannot be deformed arbitrarily close to a point on $\Sur$ or in $\Mar$.
An arc corresponds to a geodesic decomposing curve of type (3) above.
\end{defn}

\begin{defn}
A \emph{loop} is an arc whose endpoints are identified.  There are two types of loops:  A $1$-loop can be deformed to a boundary circle of $\Sur$ having no marked points, associated to a decomposing curve of type (1) above.  Those belonging to curves of type (2) are called $2$-loops.  
\end{defn}

We consider isotopy classes of arcs and loops.  Two arcs or loops are \emph{compatible} if there are curves in their respective isotopy classes which do not intersect.  
The weighting of geodesics based on their Fenchel-Nielsen coordinates now extends to weights assigned to arcs and loops:  Every arc and 1-loop has weight one, and a 2-loop has weight two.

Figure~\ref{f:legal} provides examples of marked bordered Riemann surfaces (all of genus 0).  Parts (a~--~c) show examples of compatible arcs and loops.  Part (d) shows examples of arc and loops that are not allowed.  Here, either they are intersecting the marked point set $\Mar$ or they are trivial arcs and loops, which can be deformed arbitrarily close to a point on $\Sur$ or in $\Mar$.  We bring up this distinction to denote the combinatorial difference between our situation and the world of arc complexes, recently highlighted by Fomin, Shapiro, and Thurston \cite{fst}.

\begin{figure}[h]
\includegraphics{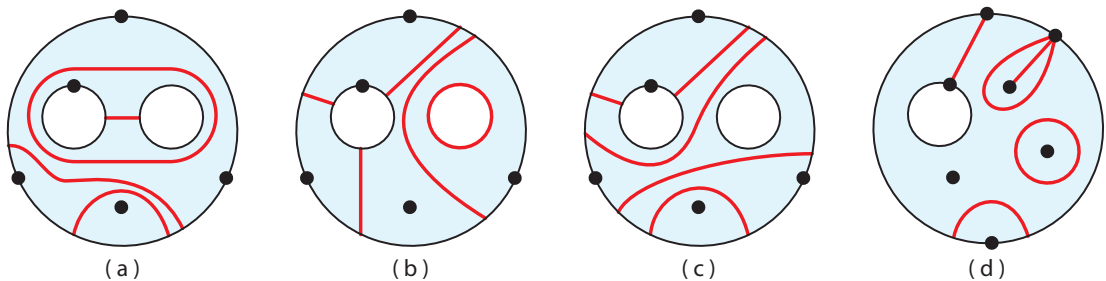}
\caption{Parts (a -- c) show compatible arcs and loops, whereas (d) shows arcs and loops that are not allowed.}
\label{f:legal}
\end{figure}

%Let $S$ be a surface with $\chi(S) < 0$. Consider the notion of a pants decomposition of $S$, which is a maximal collection of pairwise disjoint isotopy classes of essential simple closed curves in $S$ . If we cut a surface along the curves of a pants decomposition, the result is a disjoint union of pairs of pants.  By an Euler characteristic count, we see that the number of pants is $- \chi(S)$. If $S$ is a closed surface of genus $g \geq 2$, a pants decomposition of $S$ has $3g-3$ curves (each of the $2g-2$ pair of pants has three boundaries, and these curves are glued in pairs). 3.5 in Farb book.

%%%%%%%%%%%%%%%%%%%%%%%%%%%%%%%%%%%%%%%%%%%%%%%%%%%%%%%%%%%%%%%%%%%%%%%%%%%%%%%%%%%%%
%
%                Moduli space and some examples
%
%%%%%%%%%%%%%%%%%%%%%%%%%%%%%%%%%%%%%%%%%%%%%%%%%%%%%%%%%%%%%%%%%%%%%%%%%%%%%%%%%%%%%
\section{The moduli space} \label{s:exmps}
\subsection{}

Given the definition of a marked bordered Riemann surfaces, we are now in position to study its moduli space.  For a bordered Riemann surface $\Sur$ of type $(g,h)$ with marking set $\Mar$ of type $(n, \bold m)$, we denote $\Mod_{(g,h)(n, \bold m)}$ as its compactified moduli space.  Analytic methods can be used for the construction of this moduli space which follow from several important, foundational cases. 
Indeed, the moduli space $\Mod_{(0,0)(n,0)}$ is simply the classic Deligne-Mumford-Knudsen space $\Mod_{0,n}$ coming from GIT quotients.
The topology of the moduli space $\Mod_{g,n}$ of algebraic curves of genus $g$ with $n$ marked points was provided by Abikoff \cite{ab}.   Later, Sepp\"{a}l\"{a} gave a topology for the moduli space of \emph{real} algebraic curves \cite{sep}.  Finally, Liu modified this for marked bordered Riemann surfaces;  the reader is encouraged to consult \cite[Section 4]{liu} for a detailed treatment of the construction of $\Mod_{(g,h)(n, \bold m)}$.

\begin{thm} \cite{liu} \label{t:dim}
The moduli space $\Mod_{(g,h)(n, \bold m)}$ of marked bordered Riemann surfaces is equipped with a (Fenchel-Nielson) topology which is Hausdorff.  The space is compact and orientable with real dimension $6g + 3h - 6 + 2n + m.$
\end{thm}

The dimension of this space should be familiar:  It is the total weight of the decomposing curves of the surface given in Equation~\eqref{e:dim}.  Indeed, the stratification of this space is given by collections of compatible arcs and loops (corresponding to decomposing geodesics) on $(\Sur, \Mar)$ where a collection of geodesics of total weight $k$ corresponds to a (real) codimension $k$ stratum of the moduli space.
The compactification of this space is obtained by the \emph{contraction} of the arcs and loops --- the degeneration of a decomposing geodesic $\gamma$ as its length collapses to $\gamma = 0$.  

There are four possible results obtained from a contraction; each one of them could be viewed naturally in the complex double setting.  The first is the contraction of an arc with endpoints on the same boundary, as shown by an example in Figure~\ref{f:contract-arc}. 
\begin{figure}[h]
\includegraphics{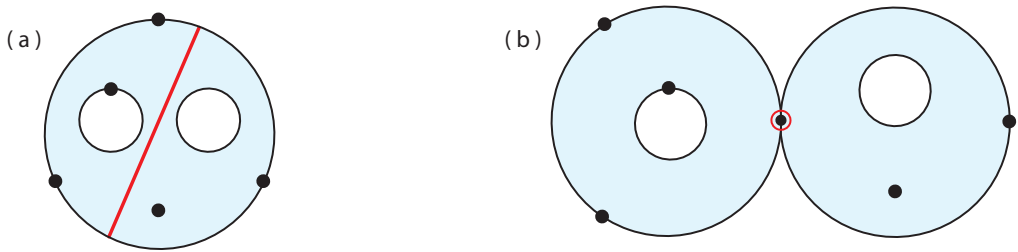}
\caption{Contraction of an arc with endpoints on the same boundary.}
\label{f:contract-arc}
\end{figure}
The arc of part (a) becomes a double point on the boundary in (b), being shared by the two surfaces.\footnote{
We abuse terminology slightly by sometimes refering to these double points as marked points.}
Similarly, one can have a contraction of an arc with endpoints on two distinct boundary components as pictured in Figure~\ref{f:contract-arc-2}.  In both these cases, we are allowed to \emph{normalize} this degeneration by pulling off (and doubling) the nodal singularity, as in Figure~\ref{f:contract-arc-2}(c).
\begin{figure}[h]
\includegraphics{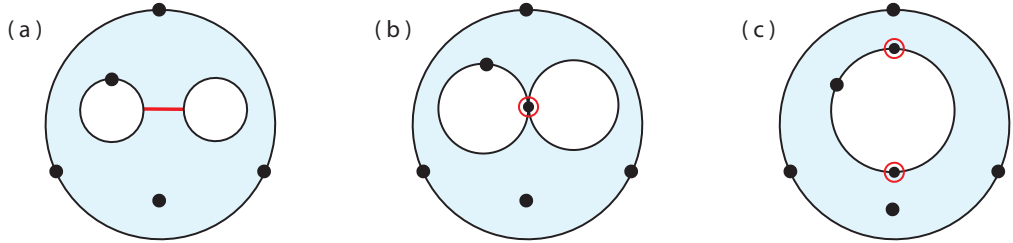}
\caption{Contraction of an arc with endpoints on distinct boundary components.}
\label{f:contract-arc-2}
\end{figure}
The other two possibilities of contraction lie with loops, as displayed in Figure~\ref{f:contract-loop}.  Part (a) shows the contraction of a 1-loop collapsing into a puncture, whereas part (b) gives the contraction of a 2-loop resulting in the classical notion of \emph{bubbling}.

\begin{figure}[h]
\includegraphics{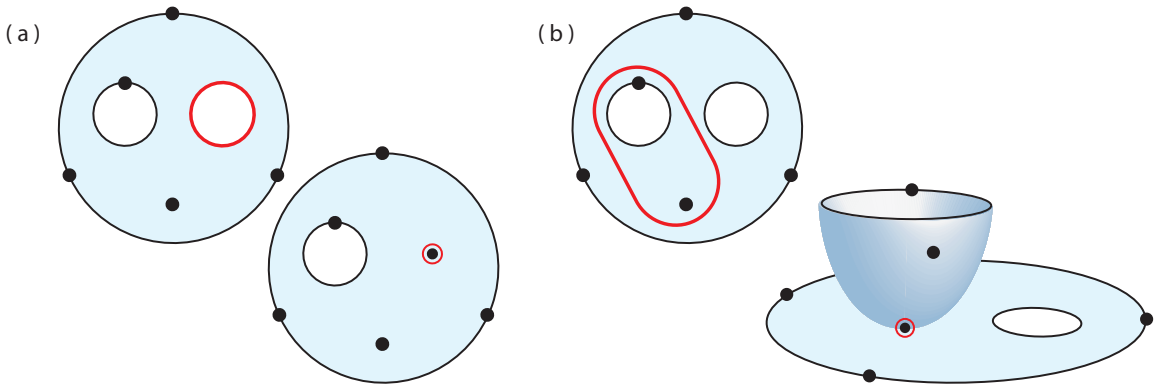}
\caption{Contraction of (a) the 1-loop and (b) the 2-loop.}
\label{f:contract-loop}
\end{figure}

%%%%%%%%%%%%%%%%%%%%%%%%%%%%%%%%%%%%%%%%%%%%%%%%%%%%%%%%%%%%%%%%%%%%%%%%%%%%%%%%%%%%%
\subsection{}

We presents several examples of low-dimensional marked bordered moduli spaces.

\begin{exmp}
Consider the moduli space $\Mod_{(0,1)(2,0)}$ of a disk with two punctures.  By Theorem~\ref{t:dim}, the dimension of this space is one.  There are only two boundary strata, one with an arc and another with a loop, showing this space to be an interval.  Figure~\ref{f:m0120} displays this example.
\begin{figure}[h]
\includegraphics{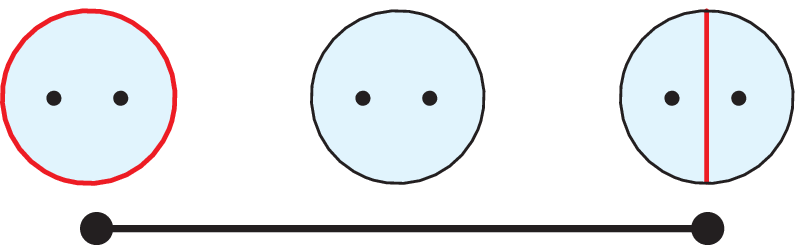}
\caption{The space $\Mod_{(0,1)(2,0)}$ is an interval.}
\label{f:m0120}
\end{figure}
The left endpoint is the moduli space of a thrice-punctured sphere, whereas the right endpoint is the product of the moduli spaces of  punctured disks with a marked point on the boundary.
\end{exmp}

%%%%%%%%%%%%%%%%%%%%%%%%%%%%%%%%%%%%%%%%%%%%%%%%%%%%%%%%%%%%%%%%%%%%%%%%%%%%%%%%%%%%%

\begin{exmp}
The moduli space $\Mod_{(0,1)(1,2)}$ can be seen as a topological disk with five boundary strata.  It has three vertices and two edges, given in Figure~\ref{f:m0112}, and is two-dimensional due to Theorem~\ref{t:dim}.  Note that the central vertex in this disk shows this space is \emph{not} a CW-complex.  The reason for this comes from the \emph{weighting} of the geodesics.  Since the loop drawn on the surface representing the central vertex has weight two, the codimension of this strata increases by two.

\begin{figure}[h]
\includegraphics{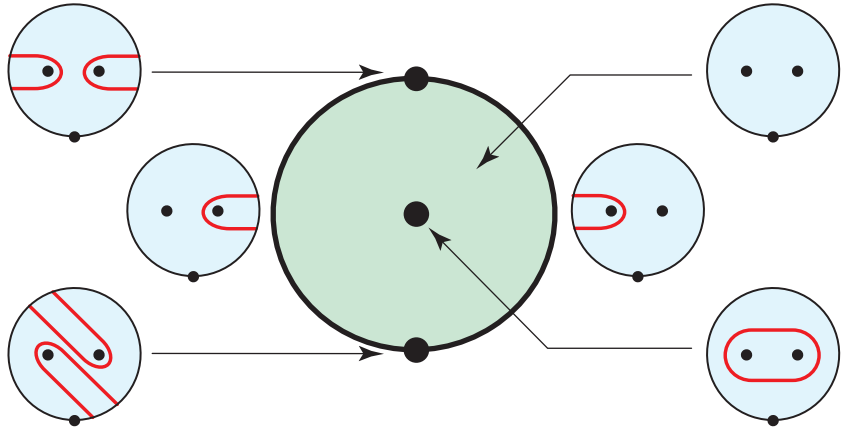}
\caption{The space $\Mod_{(0,1)(1,2)}$ is a topological disk.}
\label{f:m0112}
\end{figure}

Figure~\ref{f:twovertices} shows why there are two distinct vertices on the boundary of this moduli space.  Having drawn one arc cuts the boundary of the disk into two pieces, as labeled in (a).  The two ways of adding a second arc, shown in (b) and (c), results in the two vertices. Note that this simply amounts to relabeling the collision of the punctures with the boundary as given by (d) and (e).
\end{exmp}

\begin{figure}[h]
\includegraphics{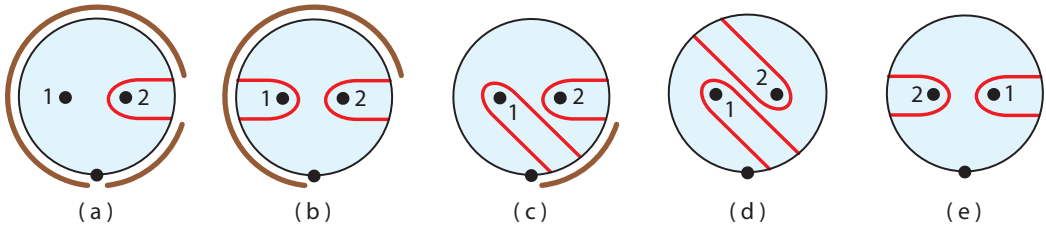}
\caption{The two vertices on the boundary of  $\Mod_{(0,1)(1,2)}$.}
\label{f:twovertices}
\end{figure}

%%%%%%%%%%%%%%%%%%%%%%%%%%%%%%%%%%%%%%%%%%%%%%%%%%%%%%%%%%%%%%%%%%%%%%%%%%%%%%%%%%%%%

\begin{exmp}
Thus far, we have considered moduli spaces of surfaces with one boundary component.  Figure~\ref{f:m02011} shows an example with two boundary components, the space $\Mod_{(0,2)(0,\la 1,1 \ra)}$.  
Like Figure~\ref{f:m0112}, the space is a topological disk, though with a different stratification of boundary pieces.  It too is not a CW-complex, again due to a loop of weight two.
\end{exmp}

\begin{figure}[h]
\includegraphics{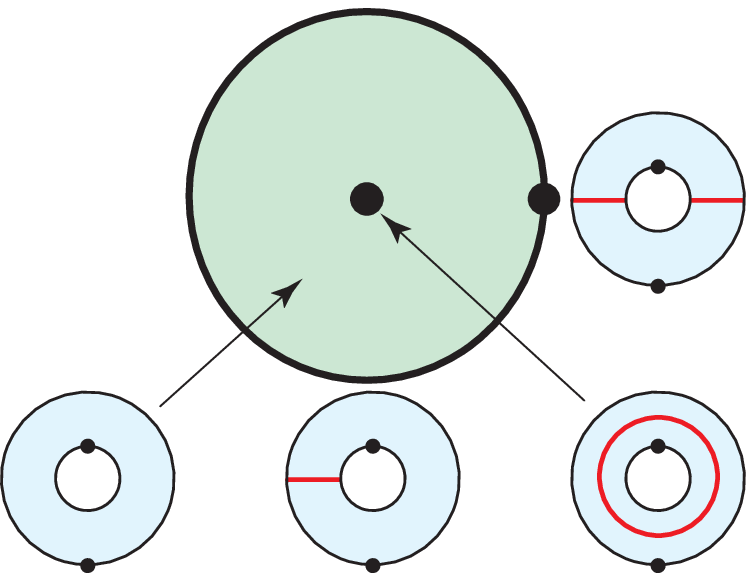}
%\includegraphics[width=\textwidth]{mod-02-03.eps}
%\caption{(a) The space $\Mod_{(0,2)(0,\la 1,1 \ra)}$ is a topological disk and (b) $\Mod_{(0,2)(1,0)}$ is a pentagon.}
\caption{The space $\Mod_{(0,2)(0,\la 1,1 \ra)}$ is a topological disk.}
\label{f:m02011}
\end{figure}

%%%%%%%%%%%%%%%%%%%%%%%%%%%%%%%%%%%%%%%%%%%%%%%%%%%%%%%%%%%%%%%%%%%%%%%%%%%%%%%%%%%%%

\begin{exmp}
The moduli space $\Mod_{(0,2)(1,0)}$ is the pentagon, as pictured in Figure~\ref{f:m0210}.  As we will see below, this pentagon can be reinterpreted as the 2D associahedron.
\end{exmp}

\begin{figure}[h]
\includegraphics{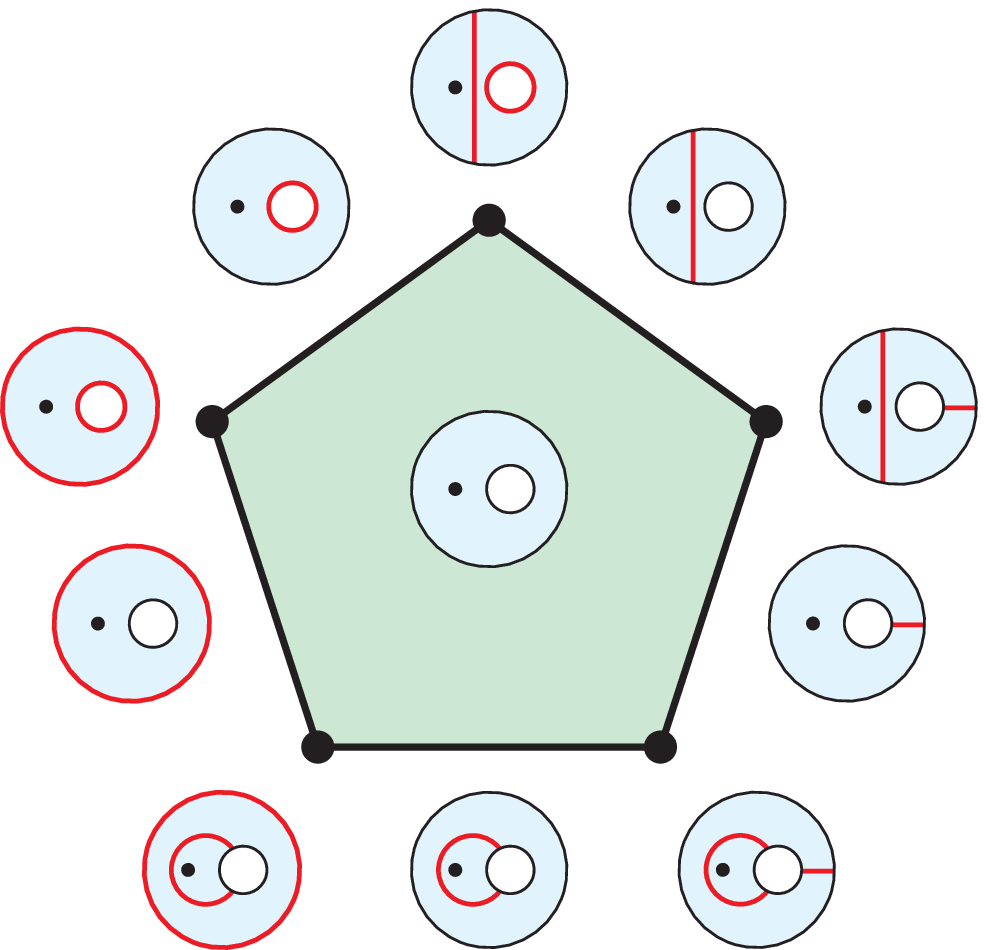}
\caption{The space $\Mod_{(0,2)(1,0)}$ is the 2D associahedron.}
\label{f:m0210}
\end{figure}

%%%%%%%%%%%%%%%%%%%%%%%%%%%%%%%%%%%%%%%%%%%%%%%%%%%%%%%%%%%%%%%%%%%%%%%%%%%%%%%%%%%%%

\begin{exmp}
By Theorem~\ref{t:dim}, the moduli space $\Mod_{(0,3)(0,0)}$ is three-dimensional.  Figure~\ref{f:m0300} shows the labeling of this space, seen as a polyhedron with three quadrilaterals and six pentagons.  Similar to above, this can viewed as 3D associahedron.
\end{exmp}

\begin{figure}[h]
\includegraphics{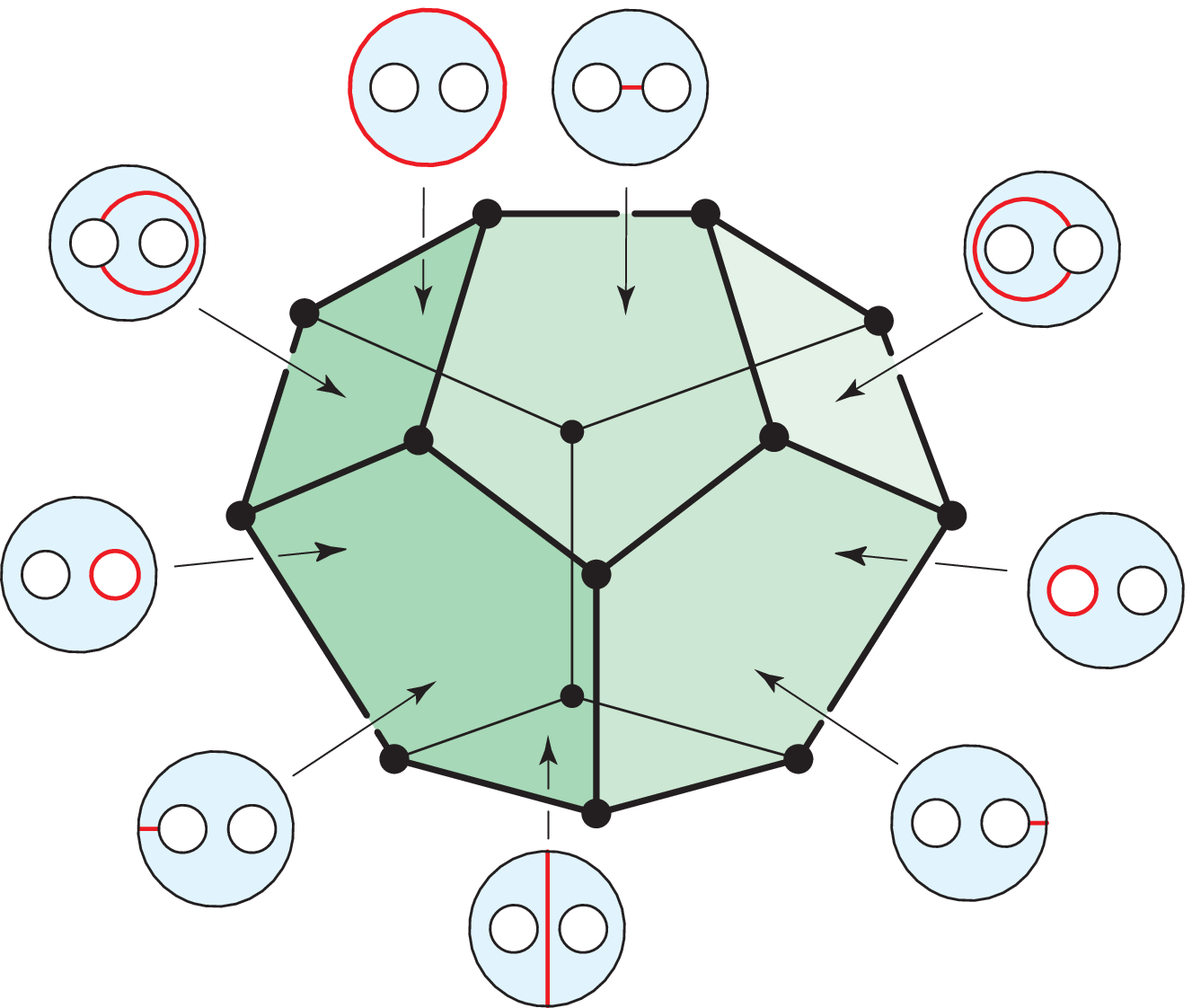}
\caption{The space $\Mod_{(0,3)(0,0)}$ is the 3D associahedron.}
\label{f:m0300}
\end{figure}

%%%%%%%%%%%%%%%%%%%%%%%%%%%%%%%%%%%%%%%%%%%%%%%%%%%%%%%%%%%%%%%%%%%%%%%%%%%%%%%%%%%%%
%
%                Convex Polytopes
%
%%%%%%%%%%%%%%%%%%%%%%%%%%%%%%%%%%%%%%%%%%%%%%%%%%%%%%%%%%%%%%%%%%%%%%%%%%%%%%%%%%%%%
\section{Convex Polytopes} \label{s:convex}

\subsection{}
From the previous examples, we see some of these moduli spaces are convex polytopes whereas others fail to even be CW-complexes.  As we see below, those with polytopal structures are all examples of the associahedron polytope.

\begin{defn}
Let $A(n)$ be the poset of all diagonalizations of a convex $n$-sided polygon, ordered such that $a \prec a'$ if $a$ is obtained from $a'$ by adding new noncrossing diagonals.  The \emph{associahedron} $K_{n-1}$ is a convex polytope of dimension $n-3$ whose face poset is isomorphic to $A(n)$.
\end{defn}

\noindent
The associahedron was originally defined by Stasheff for use in homotopy theory in connection with associativity properties of $H$-spaces \cite{jds}.  The vertices of $K_n$ are enumerated by the Catalan numbers $C_{n-1}$  and its construction as a polytope was given independently by Haiman (unpublished) and Lee \cite{lee}.

\begin{prop} \label{p:4assoc}
Four associahedra appear as moduli spaces of marked bordered surfaces:
\begin{enumerate}
\item $\Mod_{(0,0)(3,0)}$ is the 0D associahedron $K_2$.
\item $\Mod_{(0,1)(2,0)}$ is the 1D associahedron $K_3$.
\item $\Mod_{(0,2)(1,0)}$ is the 2D associahedron $K_4$.
\item $\Mod_{(0,3)(0,0)}$ is the 3D associahedron $K_5$.
\end{enumerate}
\end{prop}

\begin{proof}
The space $\Mod_{(0,0)(3,0)}$ is simply a point:  it is a pair of pants, thus trivially having a unique pair of pants decomposition.  The remaining three cases  follow from the examples depicted above in  Figures~\ref{f:m0120}, \ref{f:m0210} and \ref{f:m0300} respectively.  Figure~\ref{f:4assoc} summarizes the four surfaces encountered.
\end{proof}

\begin{figure}[h]
\includegraphics{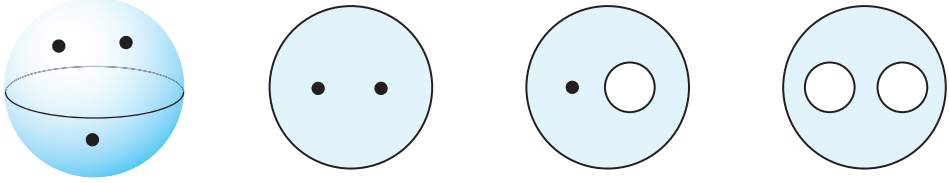}
\caption{Four surfaces whose moduli spaces yield associahedra given in Proposition~\ref{p:4assoc}.}
\label{f:4assoc}
\end{figure}

It is natural to ask which other marked bordered moduli spaces are convex polytopes, carrying a rich underlying combinatorial structure similar to the associahedra.  In other words, we wish to classify those moduli spaces whose stratifications are identical to the face posets of polytopes. We start with the following result:

\begin{thm} \label{t:polytope}
The following moduli spaces are polytopal.
\begin{enumerate}
\item $\Mod_{(0,1)(0,m)}$ with $m$ marked points on the boundary of a disk.
\item $\Mod_{(0,1)(1,m)}$ with $m$ marked points on the boundary of a disk with a puncture.
\item $\Mod_{(0,2)(0, \la m, 0 \ra)}$ with $m$ marked points on one boundary circle of an annulus.
\end{enumerate}
\end{thm}

\noindent
The first two cases will be proven below in Propositions~\ref{p:assoc} and~\ref{p:cyclo}, where they are shown to be the \emph{associahedron} and \emph{cyclohedron} respectively, and their surfaces are depicted in Figure~\ref{f:polytope-moduli}(a) and (b).  We devote a later section to prove the third case in Theorem~\ref{t:moduli-cube}, a new convex polytope called the \emph{halohedron}, whose surface is given in part (c) of the figure.
It turns out that these are the only polytopes appearing as moduli of marked bordered surfaces.

\begin{figure}[h]
\includegraphics{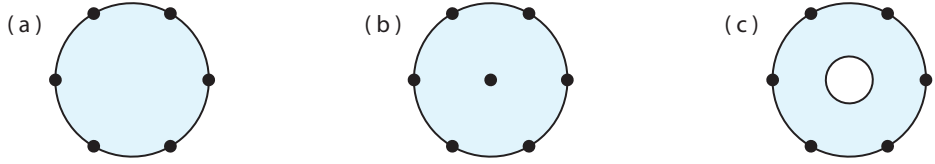}
\caption{(a) Associahedra, (b) cyclohedra, (c) halohedra.}
\label{f:polytope-moduli}
\end{figure}

\begin{rem}
Parts (1) and (2) of Theorem~\ref{t:polytope} can be reinterpreted as types $A$ and $B$ generalized associahedra of Fomin and Zelevinsky \cite{fz}.   Therefore, it is tempting to think that this new polytope in part (3) could be the type $D$ version.  This is, however, not the case, as we see below.  In particular, the 3D type $D$ generalized associahedron is simply the classical associahedron, whereas the 3D polytope of part (3) is different, as displayed on the right side of Figure~\ref{f:cube-y3}.
\end{rem}
 
\begin{thm} \label{t:notpolytope}
All other moduli spaces of stable bordered marked surfaces not mentioned in Proposition~\ref{p:4assoc} and Theorem~\ref{t:polytope} are not polytopes.
\end{thm}

\begin{proof}
The overview of the proof is as follows:  We find certain moduli spaces which do not have polytopal structures, and then show these moduli spaces appearing as lower dimensional strata to the list above.   Since polytopes must have polytopal faces, this will show that none of the spaces on the list are polytopal.

Outside of the spaces in  Proposition~\ref{p:4assoc} and Theorem~\ref{t:polytope}, the cases of $\Mod_{(g,h)(n, \bold m)}$ which result in \emph{stable} spaces that we must consider are the following:
\begin{enumerate}
\item when $g > 0$;
\item when $h+n > 3$;
\item when $h+n=3$ and some $m_i > 0$;
\item when $h=2$, and both $m_1>0$ and $m_2>0$.
\end{enumerate}
We begin with item (1), when $g > 0$.  The simplest case to consider is that of a torus with one puncture.  (A torus with no punctures is not stable, with infinite automorphim group, and so cannot be considered.)  The moduli space $\Mod_{(1,0)(1,0)}$ of a torus with one puncture is two-dimensional (from Theorem~\ref{t:dim}) and has the stratification of a 2-sphere, constructed from a 2-cell and a 0-cell.  Thus, any other surface of non-zero genus, regardless of the number of marked points or boundary circles will always have $\Mod_{(1,0)(1,0)}$ appearing in its boundary strata; Figure~\ref{f:torusposet} shows such a case where this punctured torus appears in the boundary of $\Mod_{(2,3)(1,0)}$.   This immediately eliminates all $g>0$ cases from being polytopal.  For the remaining cases, we need only focus on the $g=0$ case.

\begin{figure}[h]
\includegraphics{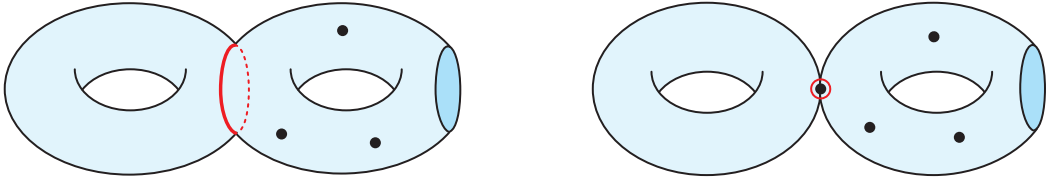}
\caption{The appearance of $\Mod_{(1,0)(1,0)}$ in boundary of higher genus moduli.}
\label{f:torusposet}
\end{figure}

Now consider item (2), when $h+n > 3$.  When $h=0$ and $n>3$, the resulting moduli space is exactly the classical compactified moduli space \CM{n} of punctured Riemann spheres.  The strata of this space is well-known, indexed by labeled trees, and is not a polytope.  For all other values of $h+n >3$, where $h>0$, the space $\Mod_{(0,1)(1,2)}$ appears as a boundary strata.  As we have shown in Figure~\ref{f:m0112}, this space is not a polytope.  Similary, for items (3) and (4), 
%when $h+n=3$ and $m_1 > 0$, 
the moduli space $\Mod_{(0,2)(0,\la 1,1\ra)}$ appears in all strata, with Figure~\ref{f:m02011} showing the polytopal failure.  
%Finally, note item (4), when $h=2$, $m_1>0$ and $m_2>0$. In this case, the space $\Mod_{(0,1)(1,2)}$ given in Figure~\ref{f:m0112} appears in the strata.
\end{proof}

\begin{rem}
An alternate understanding of this result lies in 2-loops.  Since each 2-loop has weight two, Lemma~\ref{l:weight} and Theorem~\ref{t:dim} show that any such loop will increase the codimension  of the corresponding strata by \emph{two}.  This automatically disqualifies the stratification from being that of a polytope.  Indeed, the cases outlined in Theorem~\ref{t:notpolytope} are exactly those which allow 2-loops.
\end{rem}

\hide{%%%%%%%%%%%%%%%
h+n = 3
h=0 n=3 (point)
h=1 n=2  m1=0 (interval)
h=2 n =1 m1=0 m2=0 (pentagon)
h=3 n=0 m1=m2=m3=0 (3d assoc)

h+n = 2
h=0 n =2 (not stable)
h=1 n=1  m1=POS (nD cyclo ***)
h=2 n=0  m1=POS m2=0 (nd cubea ***)
h=2 n=0  m1=0 m2=0 (not stable)

h+n =1
h=0 n=1 (not stable)
h=1 n=0 m1=POS (nD assoc ***)
}%%%%%%%%%%%%%%%

%%%%%%%%%%%%%%%%%%%%%%%%%%%%%%%%%%%%%%%%%%%%%%%%%%%%%%%%%%%%%%%%%%%%%%%%%%%%%%%%%%%%%
\subsection{} \label{s:polyass}

We close this section with two well-known results.
%; we include sketches of their proof here for the sake of completeness.

\begin{prop} \cite[Section 2]{dev1} \label{p:assoc}
The space $\Mod_{(0,1)(0,m)}$ of $m$ marked points on the boundary of a disk is the associahedron $K_{m-1}$.
\end{prop}

\begin{proof}[Sketch of Proof]
%Recall that the associahedra $K_{m-1}$ has the poset structure of ways to add noncrossing diagonals to an $m$-sided polygon.  
Construct a dual between $m$ marked points on the boundary of a disk and an $m$-gon by identifying each marked point to an edge of the polygon, in cyclic order.   Then each arc on the disk corresponds to a diagonal of the polygon, and since arcs are compatible, the diagonals are noncrossing.  
Figure~\ref{f:bijection-k4} shows an example for the 2D case.
\end{proof}

\begin{figure}[h]
\includegraphics{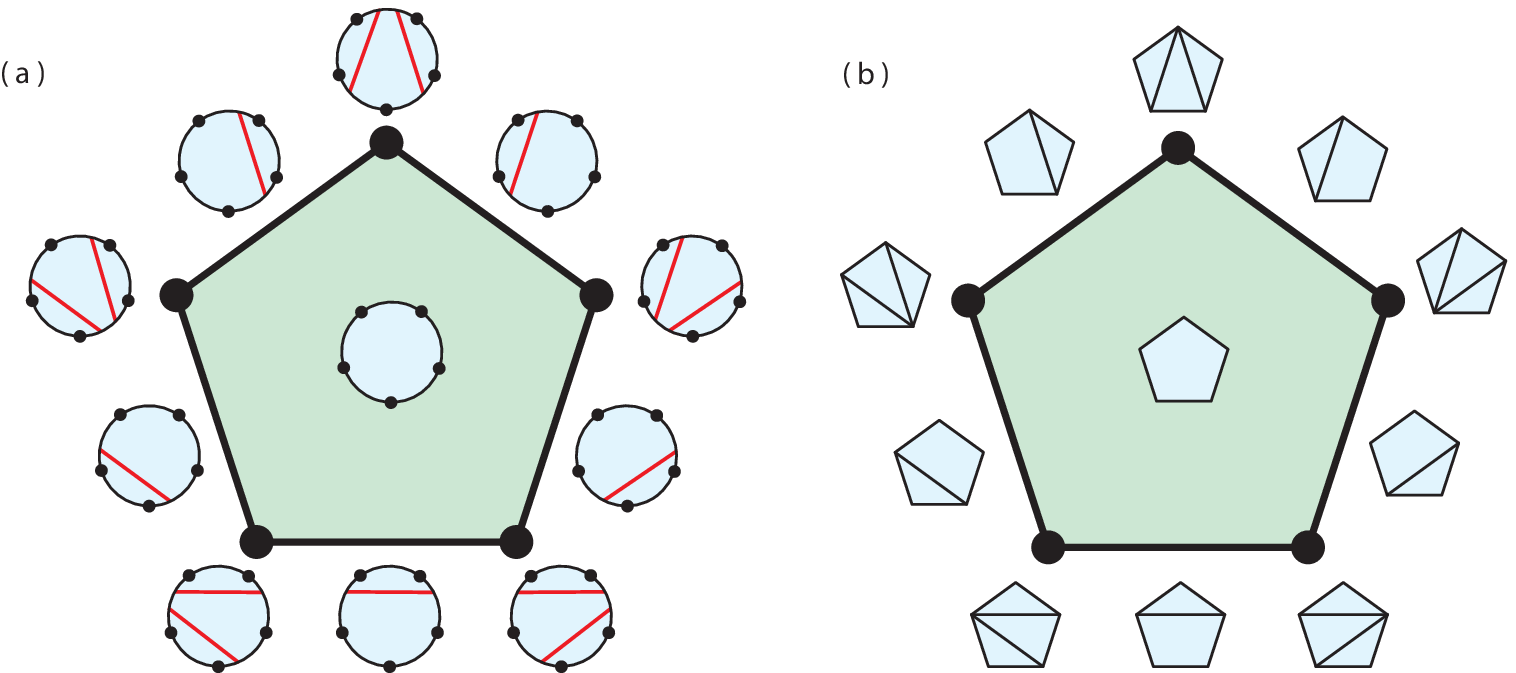}
\caption{The bijection between $\Mod_{(0,1)(0,5)}$ and $K_4$.}
\label{f:bijection-k4}
\end{figure}

\begin{rem}
The appearance of the associahedron as the moduli space $\Mod_{(0,1)(0,m)}$ is quite natural.  Indeed, if we considered all ordered labelings of the $m$ marked points on the boundary of the disk, one would obtain \M{m}, the \emph{real} points of the Deligne-Mumford space \CM{m} tiled by $(m-1)!/2$ associahedra \cite{dev1}.  \end{rem}

Figure~\ref{f:holomorphic} shows (a) the complex double $\Sur_\C$ with $m$ marked points on the boundary fixed by involution.  Part (b) shows the surface $\Sur$ with boundary, whereas (c) considers this space with the natural hyperbolic structure, where the marked points become cusps in the plane.

\begin{figure}[h]
\includegraphics{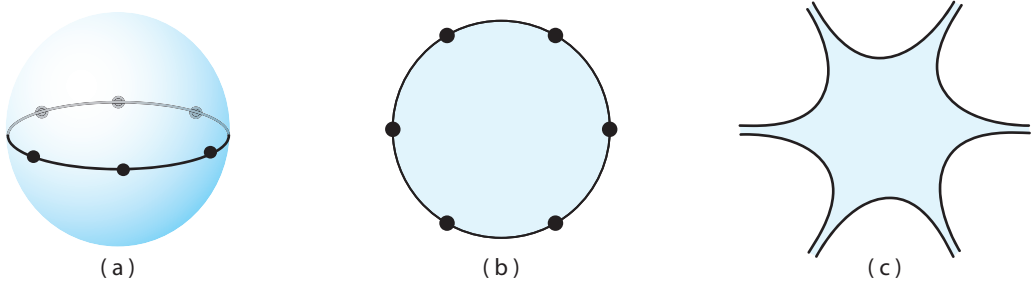}
\caption{(a) The complex double, (b) the surface with boundary, and (c) its hyperbolic structure with planar cusps.}
\label{f:holomorphic}
\end{figure}

%%%%%%%%%%%%%%%%%%%%%%%%%%%%%%%%%%%%%%%%%%%%%%%%%%%%%%%%%%%%%%%%%%%%%%%%%%%%%%%%%%%%%
\subsection{} \label{s:polycyc}

A close kin to the associahedron is the \emph{cyclohedron}. This polytope originally manifested in the work of Bott and Taubes \cite{bt} with respect to knot invariants and later given its name by Stasheff. 

\begin{defn}
Let $B(n)$ be the poset of all diagonalizations of a convex $2n$-sided centrally symmetric polygon, ordered such that $a \prec a'$ if $a$ is obtained from $a'$ by adding new noncrossing diagonals. Here, a diagonal will either mean  a {\em pair} of centrally symmetric diagonals or a \emph{diameter} of the polygon. The \emph{cyclohedron} $W_n$ is a convex polytope of dimension $n-1$ whose face poset is isomorphic to $B(n)$.
\end{defn}

\begin{prop} \cite[Section 1]{dev2}  \label{p:cyclo}
The space $\Mod_{(0,1)(1,m)}$ of $m$ marked points on the boundary of a disk with a puncture is the cyclohedron $W_m$.
\end{prop}

\begin{proof}[Sketch of Proof]
Construct a dual to the $m$ marked points on the boundary of the disk to the symmetric $2m$-gon by identifying each marked point to a pair of antipodal edges of the polygon, in cyclic order.   Then each arc on the disk corresponds to a pair of symmetric diagonals of the polygon; in particular, the arcs which partition the puncture from the boundary points map to the diameters of the polygon.  Since arcs are compatible, the diagonals are noncrossing.  Figure~\ref{f:bijection-w3} shows an example for the 2D case.
\end{proof}

\begin{figure}[h]
\includegraphics {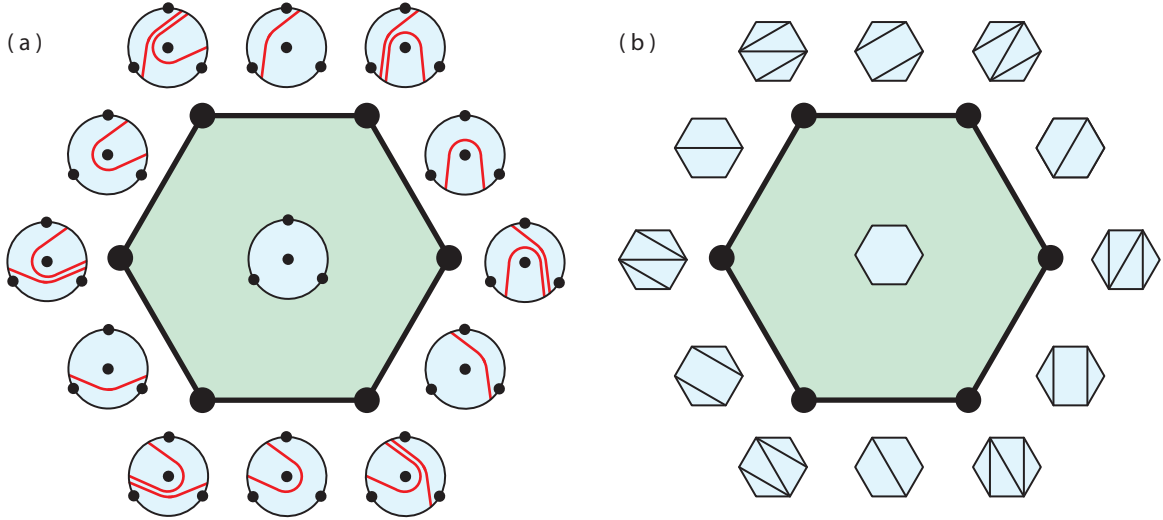}
\caption{Cyclohedron $W_3$ using brackets and polygons.}
\label{f:bijection-w3}
\end{figure}

\begin{rem}

For associahedra, the moduli space of $m$ marked points on the boundary of a disk, any arc must have at least two marked points on either side of the partition in order to maintain stability.   For the cyclohedra, however, the puncture in the disk allows us to bypass this condition since a punctured disk with one marked point on the boundary is stable.
From the perspective of polytopes, this translates into different stratification of the faces:  Whereas each face of $K_n$ is a product of associahedra, resulting in an operad structure, faces of $W_n$ are products of associahedra and cyclohedra, yielding a module over an operad. 
\end{rem}

%%%%%%%%%%%%%%%%%%%%%%%%%%%%%%%%%%%%%%%%%%%%%%%%%%%%%%%%%%%%%%%%%%%%%%%%%%%%%%%%%%%%%
%
%                Graph Associahedra and Truncations of Cubes
%
%%%%%%%%%%%%%%%%%%%%%%%%%%%%%%%%%%%%%%%%%%%%%%%%%%%%%%%%%%%%%%%%%%%%%%%%%%%%%%%%%%%%%
\section{Graph Associahedra and Truncations of Cubes}  
\label{s:cubes}
\subsection{}

This section focuses on understanding the polytope arising from the moduli space of marked points on an annulus, as shown in Figure~\ref{f:polytope-moduli}(c).  However, we wish to place this polytope in a much larger context based on truncations of simplices and cubes.  We begin with motivating definitions of \emph{graph associahedra}; the reader is encouraged to see \cite[Section 1]{cd} for details.

\begin{defn}
Let $G$ be a connected graph.  A \emph{tube} is a set of nodes of $G$ whose induced graph is a connected 
proper subgraph of $G$.  Two tubes $u_1$ and $u_2$ may interact on the graph as follows:
\begin{enumerate}
\item Tubes are \emph{nested} if  $u_1 \subset u_2$.
\item Tubes \emph{intersect} if $u_1 \cap u_2 \neq \emptyset$ and $u_1 \not\subset u_2$ and $u_2 \not\subset u_1$.
\item Tubes are \emph{adjacent} if $u_1 \cap u_2 = \emptyset$ and $u_1 \cup u_2$ is a tube in $G$.
\end{enumerate}
Tubes are \emph{compatible} if they do not intersect and are not adjacent.  A \emph{tubing} $U$ of $G$ is a set of tubes of $G$ such that every pair of tubes in $U$ is compatible.
\end{defn}

\begin{thm} \cite[Section 3]{cd} \label{t:graph} \label{t:graphass}
For a graph $G$ with $n$ nodes, the \emph{graph associahedron} $\KG$ is a simple convex polytope of dimension $n-1$ whose face poset is isomorphic to the set of tubings of $G$, ordered such that $U \prec U'$ if $U$ is obtained from $U'$ by adding tubes.
\end{thm}

\begin{cor} \label{c:asscyc}
When $G$ is a path with $n$ nodes, $\KG$ becomes the associahedron $K_{n+1}$.  Similarly, when $G$ is a cycle with $n$ nodes, $\KG$ is the cyclohedron $W_n$. 
\end{cor}

Figure~\ref{f:kwexmp} shows the 2D examples of these cases of graph associahedra, having underlying graphs as paths and cycles, respectively, with three nodes. Compare with Figures~\ref{f:bijection-k4} and~\ref{f:bijection-w3}.

\begin{figure}[h]
\includegraphics{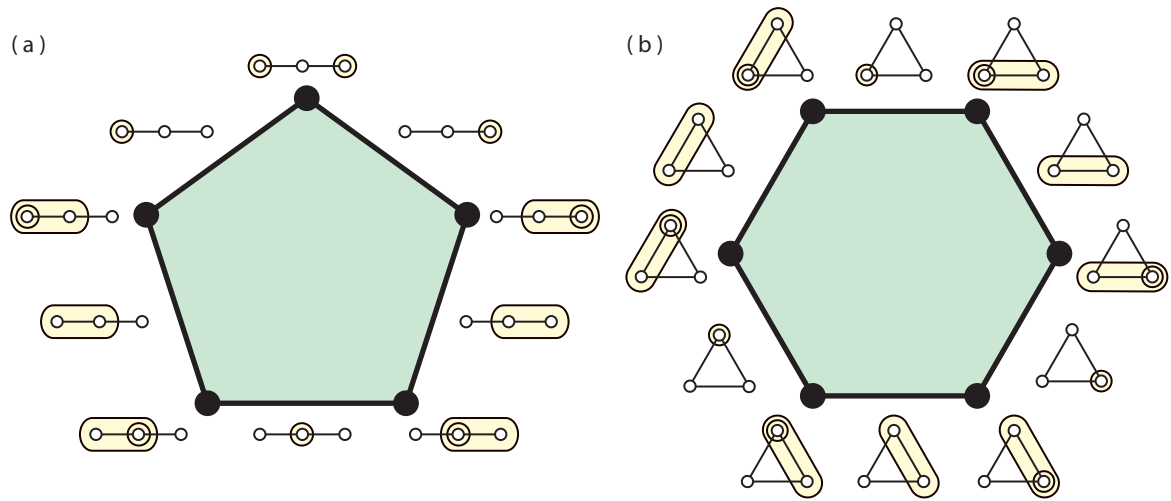}
\caption{Graph associahedra of the (a) path and (b) cycle with three nodes as underlying graphs.}
\label{f:kwexmp}
\end{figure}

There exists a natural construction of graph associahedra from iterated truncations of the simplex:  For a graph $G$ with $n$ nodes, let $\triangle_G$ be the $(n{-}1)$-simplex in which each facet (codimension one face) corresponds to a particular node.  Each proper subset of nodes of $G$ corresponds to a unique face of $\triangle_G$, defined by the intersection of the faces associated to those nodes, and the empty set corresponds to the face which is the entire polytope $\triangle_G$.

\begin{thm} \cite[Section 2]{cd} \label{t:graphasstrunc}
For a graph $G$, truncating faces of $\triangle_G$ which correspond to tubes in increasing order of dimension results in the graph associahedron $\KG$.
\end{thm}

%%%%%%%%%%%%%%%%%%%%%%%%%%%%%%%%%%%%%%%%%%%%%%%%%%%%%%%%%%%%%%%%%%%%%%%%%%%%%%%%%%%%%
\subsection{}

We now create a new class of polytope that mirror graph associahedra, except now we are interested in truncations of cubes.  We begin with the notion of design tubes.

\begin{defn}
Let $G$ be a connected graph.  A \emph{round tube} is a set of nodes of $G$ whose induced graph is a connected 
(and not necessarily proper) subgraph of $G$.  A \emph{square tube} is a single node of $G$.  Such tubes are called \emph{design tubes} of $G$.  Two design tubes are \emph{compatible} if %they do not intersect and 
\begin{enumerate}
\item when they are both round, they are not adjacent and do not intersect.
\item otherwise, they are not nested.
\end{enumerate}
A \emph{design tubing} $U$ of $G$ is a collection of design tubes of $G$ such that every pair of tubes in $U$ is compatible.  
\end{defn}

Figure~\ref{f:designtubes} shows examples of design tubings.  Note that unlike ordinary tubes, round tubes do not have to be proper subgraphs of $G$.  Based on design tubings, we construct a polytope, not from truncations of simplices but cubes:
\begin{figure}[h]
\includegraphics{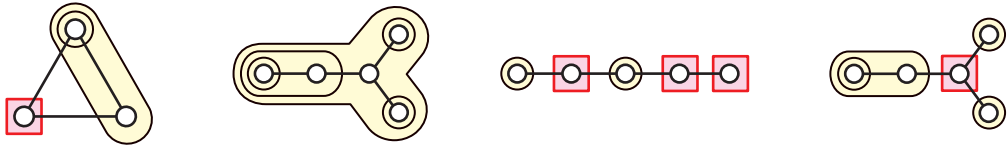}
\caption{Design tubings.}
\label{f:designtubes}
\end{figure}
For a graph $G$ with $n$ nodes, we define $\square_G$ to be the $n$-cube 
where each pair of opposite facets correspond to a particular node of $G$.  Specifically, one facet in the pair
represents that node as a round tube and the other represents it as a square tube.  Each subset of nodes of $G$, chosen to be either round or square, corresponds to a unique face of $\square_G$, defined by the intersection of the faces associated to those nodes.  The empty set corresponds to the face which is the entire polytope $\square_G$.

\begin{defn}
For a graph $G$, truncate faces of $\square_G$ which correspond to round tubes in increasing order of dimension. The resulting polytope $\CG$ is the \emph{graph cubeahedron}.
\end{defn}

\begin{exmp}
Figure \ref{f:cube-y3} displays the construction of $\CG$ when $G$ is a cycle with three nodes.  The facets of the $3$-cube are labeled with nodes of $G$, each pair of opposite facets being round or square.   The first step is the truncation of the corner vertex labeled with the round tube being the entire graph.  Then the three edges labeled by tubes are truncated.
\end{exmp}

\begin{figure}[h]
\includegraphics{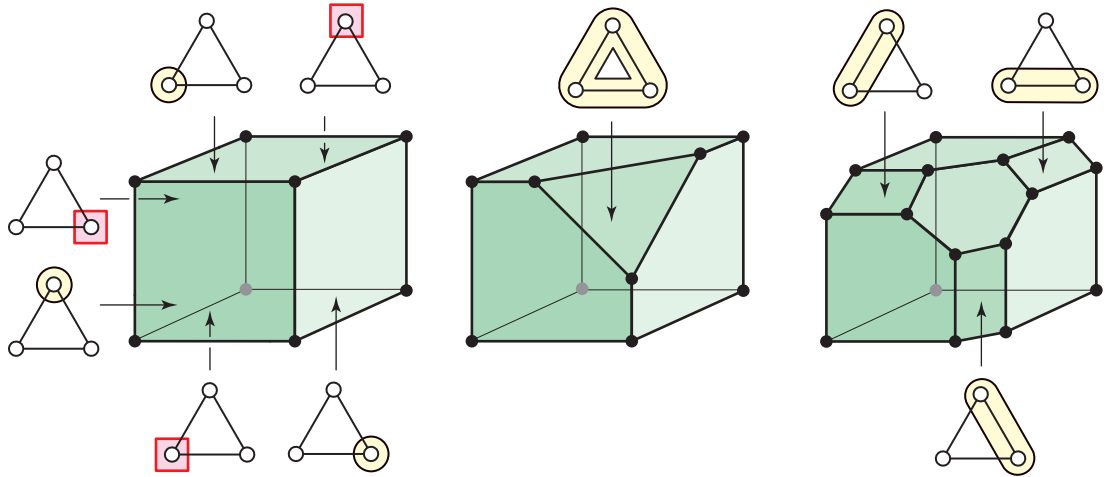}
\caption{Iterated truncations of the $3$-cube based on a 3-cycle.}
\label{f:cube-y3}
\end{figure}

\begin{thm} \label{t:graphcube}
For a graph $G$ with $n$ nodes, the graph cubeahedron $\CG$ is a simple convex polytope of dimension $n$ whose face poset is isomorphic to the set of design tubings of $G$, ordered such that $U \prec U'$ if $U$ is obtained from $U'$ by adding tubes.
\end{thm}

\begin{proof}
The initial truncation of the vertex of  $\square_G$ where all $n$ round-tubed facets intersect creates an $(n-1)${-}simplex.  The labeling of this simplex is the round tube corresponding to the entire graph $G$.  The round-tube labeled $k$-faces of $\square_G$ induce a labeling of the $(k-1)$-faces of this simplex, which can readily be seen as $\triangle_G$.  Since the graph cubeahedron is obtained by iterated truncation of the faces labeled with \emph{round tubes}, we see from Theorem~\ref{t:graphasstrunc} that this converts $\triangle_G$ into the graph associahedron $\KG$.  The $2n$ facets of $\CG$ coming from $\square_G$ are in bijection with the design tubes of $G$ capturing one vertex.  The remaining facets of $\CG$ as well as its face poset structure follow immediately from Theorem~\ref{t:graphass}.
\end{proof}

\begin{cor} \label{c:enum}
Let $G$ be a graph with $n$ vertices, and let $|\CG|$ and $|\KG|$ be the number of facets of $\CG$ and $\KG$ respectively.  Then $|\CG| = |\KG| + 1 + n.$
\end{cor}

%%%%%%%%%%%%%%%%%%%%%%%%%%%%%%%%%%%%%%%%%%%%%%%%%%%%%%%%%%%%%%%%%%%%%%%%%%%%%%%%%%%%%
\subsection{}

We are now in position to justify the introduction of graph cubeahedra in the context of moduli spaces.

\begin{thm} \label{t:moduli-cube}
Let $G$ be a cycle with $m$ nodes.  The moduli space $\Mod_{(0,2)(0, \la m, 0 \ra)}$ of $m$ marked points on one boundary circle of an annulus is the graph cubeahedron $\CG$.  We denote this special case as $\Y_m$ and call it the \emph{halohedron}.
\end{thm}

\begin{proof}
There are three kinds of boundary appearing in the moduli space.  The loop capturing the unmarked circle of the annulus is in bijection with the round tube of the entire graph, as shown in Figure~\ref{f:bi-cycle}(a).
\begin{figure}[h]
\includegraphics{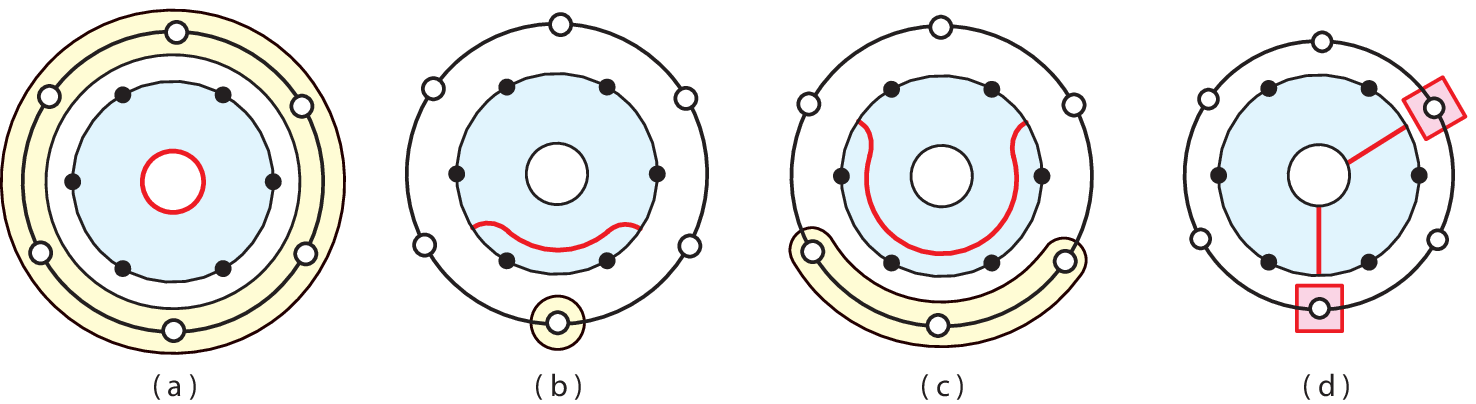}
\caption{Bijection between $\Mod_{(0,2)(0, \la m, 0 \ra)}$ and $\Y_m$.}
\label{f:bi-cycle}
\end{figure}
The arcs capturing $k$ marked boundary points correspond to the round tube surrounding $k-1$ nodes of the cycle, displayed in parts (b) and (c).  Finally, arcs between the two boundary circles of the annulus are in bijection with associated square tubes of the graph, as in (d).  The identification of these poset structures are then immediate.
\end{proof}

Figure~\ref{f:cube-y2}(a) shows the labeling of $\Y_2$ based on arc and loops on an annulus. By construction of the graph cubeahedron, this pentagon should be viewed as a truncation of the square.  The 3D halohedron $\Y_3$ is depicted on the right side of Figure~\ref{f:cube-y3}.

\begin{figure}[h]
\includegraphics[width=\textwidth]{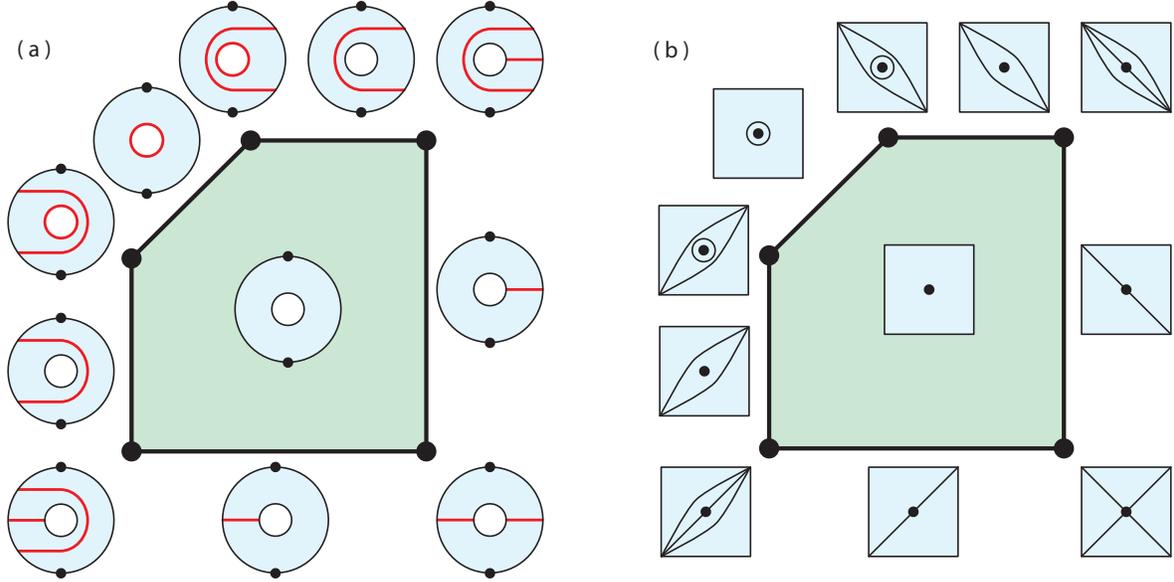}
\caption{(a) The two-dimensional $\Y_2$ and (b) its labeling with polygons.}
\label{f:cube-y2}
\end{figure}

\begin{rem}
It is natural to expect these $\Y_m$ polytopes to appear in facets of other moduli spaces.  Figure~\ref{f:pentagon-face} shows why three pentagons of Figure~\ref{f:m0300} are actually halohedra in disguise:  The arc is contracted to a marked point, which can be doubled and pulled open due to normalization.  The resulting surface is the pentagon $\Y_2$.  
\end{rem}

\begin{figure}[h]
\includegraphics{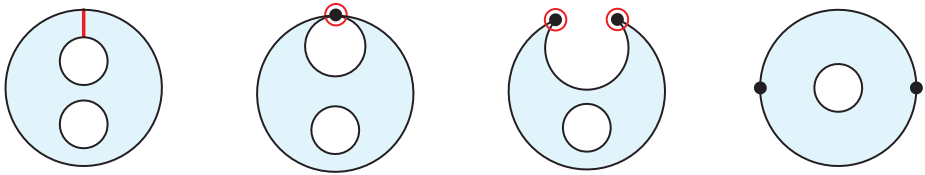}
\caption{Certain pentagons of Figure~\ref{f:m0300} are actually the polygon $\Y_2$.}
\label{f:pentagon-face}
\end{figure}

%%%%%%%%%%%%%%%%%%%%%%%%%%%%%%%%%%%%%%%%%%%%%%%%%%%%%%%%%%%%%%%%%%%%%%%%%%%%
%
%           Combinatorial and Algebraic Structures
%
%%%%%%%%%%%%%%%%%%%%%%%%%%%%%%%%%%%%%%%%%%%%%%%%%%%%%%%%%%%%%%%%%%%%%%%%%%%%
\section{Combinatorial and Algebraic Structures}
\label{s:combin}
\subsection{}

We close with examining the graph cubeahedron $\CG$ in more detail, especially in the cases when $G$ is a path and a cycle.  
We have discussed in Sections~\ref{s:polyass} and~\ref{s:polycyc} the poset structure of associahedra $K_n$ and cyclohedra $W_n$ in terms of polygons.   We now talk about the polygonal version of the halohedron $\Y_n$:  Consider a convex $2n$-sided centrally symmetric polygon with an additional central vertex, as given in Figure~\ref{f:cube-polygon}(a).  The three kinds of boundary strata which appear in $\Mod_{(0,2)(0, \la m, 0 \ra)}$ can be interpreted as adding diagonals to the symmetric polygon.  Part (b) shows that a ``circle diagonal'' can be drawn around the vertex, parts (c) and (d) show how a pair of centrally symmetric diagonals can be used, and a diameter can appear as in part (e).

\begin{figure}[h]
\includegraphics{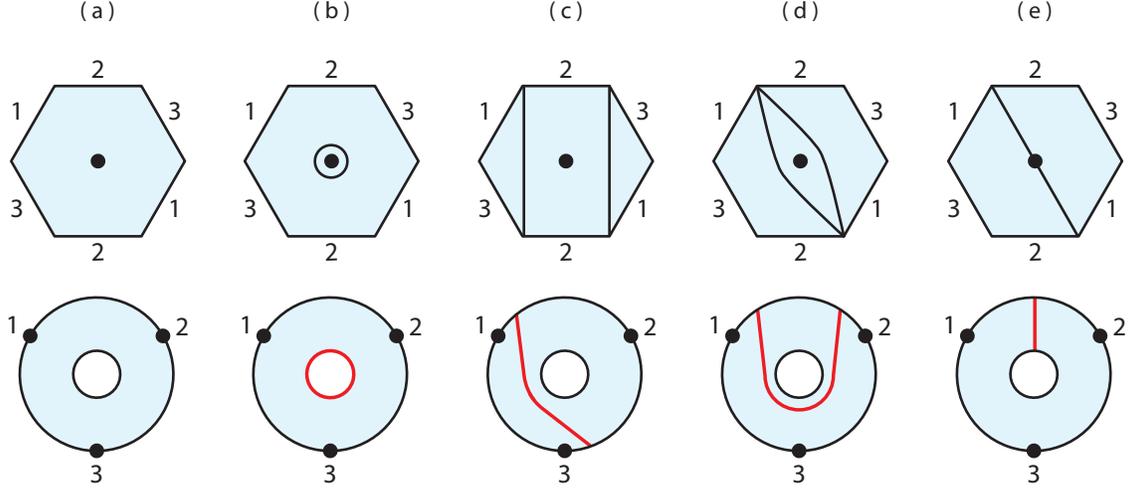}
\caption{Polygonal labeling of $\Y_n$.}
\label{f:cube-polygon}
\end{figure}

Due to the central vertex of the symmetric polygon, it is important to distinguish a pair of diameters as in (d) versus one diameter as in (e).  Indeed, the pair of diagonals are compatible with the circular diagonal, whereas the diameter is not since they intersect.  On the other hand, two diameters are compatible since they are considered not to cross due to the central vertex. 
Figure~\ref{f:cube-y2}(b) shows the case of $\Y_2$ now labeled using polygons, showing the compatibility of the different diagonals.
We summarize this below.  

\begin{defn}
Let $C(n)$ be the poset of all diagonalizations of a convex $2n$-sided centrally symmetric polygon with a central vertex, ordered such that $a \prec a'$ if $a$ is obtained from $a'$ by adding new noncrossing diagonals. Here, a diagonal will either mean a circle around the central vertex,  a pair of centrally symmetric diagonals or a diameter of the polygon. The halohedron $\Y_n$ is a convex polytope of dimension $n$ whose face poset is isomorphic to $C(n)$.
\end{defn}

%%%%%%%%%%%%%%%%%%%%%%%%%%%%%%%%%%%%%%%%%%%%%%%%%%%%%%%%%%%%%%%%%%%%%%%%%%%%
\subsection{}

Let us now turn to the case of $\CG$ when $G$ is a path:

\begin{prop} \label{p:k-path}
If $G$ is a path with $n$ nodes, then $\CG$ is the associahedron $K_{n+2}$.
\end{prop}

\begin{proof}
Let $G^*$ be a path with $n+1$ nodes.  From Corollary~\ref{c:asscyc}, the associahedron $K_{n+2}$ is the graph associahedron $\KG^*$.  We therefore show a bijection between design tubings of $G$ (a path with $n$ nodes) and regular tubings of $G^*$ (a path with $n+1$ nodes), as in Figure~\ref{f:bijection-path-assoc}.  Each round tube of $G$ maps to its corresponding tube of $G^*$.  Any square tube around the $k$-th node of $G$ maps to the tube of $G$ surrounding the $\{k+1, k+2, \ldots, n, n+1\}$ nodes.  This mapping between tubes naturally extends to tubings of $G$ and $G^*$ since the round and square tubes of $G$ cannot be nested.  The bijection follows.
\end{proof}

\begin{figure}[h]
\includegraphics{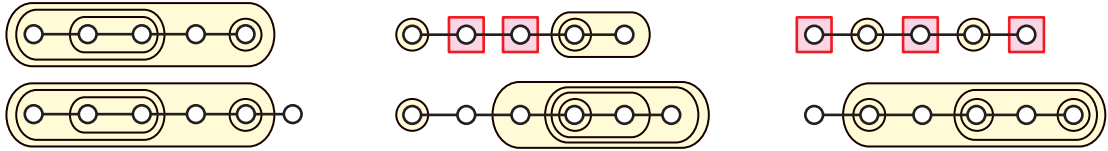}
\caption{Bijection between design tubes and regular tubes on paths.}
\label{f:bijection-path-assoc}
\end{figure}

We know from Theorem~\ref{t:graphasstrunc} that associahedra can be obtained by truncations of the simplex.  But since $\CG$ is obtained by truncations of cubes, Proposition~\ref{p:k-path} ensures that associahedra can be obtained this way as well.  Such an example is depicted in Figure~\ref{f:cube-k6}, where the 4D associahedron $K_6$ appears as iterated truncations of the 4-cube.

\begin{figure}[h]
\includegraphics{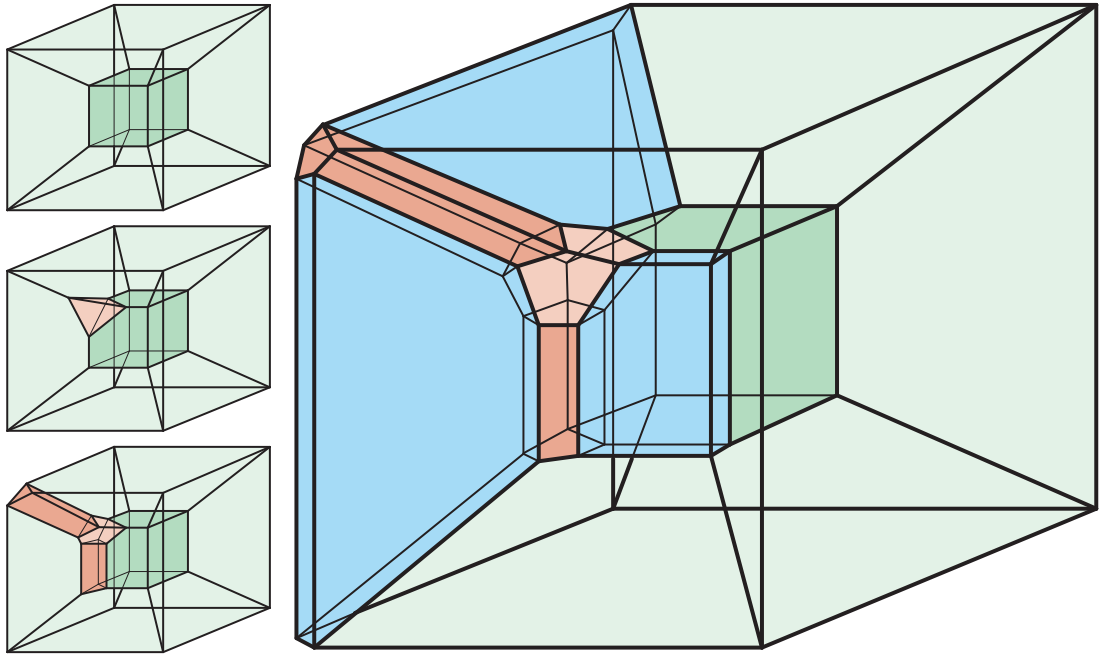}
\caption{The iterated truncation of the 4-cube resulting in $K_6$.}
\label{f:cube-k6}
\end{figure}

\hide{ %%%%%%%%%%%%%
\begin{cor}
When  $G$ is a path, we see KG and MG are both associahera, where latter is a dimension greater.
\end{cor}

\begin{proof}
$|KG| = 1/2 (n-1)(n+2)$ when $|G| =n$.
Then 
$1/2(n-1)(n+2) + 1 + n = \cdots = 1/2 (n)(n+3)$.
\end{proof}

}%%%%%%%%%%%%%%%%%%%%%%

\begin{prop}\label{t:faces}
The facets of $\Y_n$ are 
\begin{enumerate}
\item one copy of $W_n$
\item $n$ copies of $K_{n+1}$ and
\item $n^2 -n$ copies of $K_m \times \Y_{n-m+1}$.
\end{enumerate}
\end{prop}

\begin{proof}
Recall there are three kinds of codimension one boundary strata appearing in the moduli space $\Mod_{(0,2)(0, \la n, 0 \ra)}$.  First, there is the unique loop around the unmarked boundary circle, which Proposition~\ref{p:cyclo} reveals as the cyclohedron $W_n$.  Second, there exist arcs between distinct boundary circles, as show in Figure~\ref{f:bijection-path}(a).    This figure establishes a bijection between this facet and $\CG$, where $G$ is a path with $n-1$ nodes.  By Theorem~\ref{p:k-path} above, such facets are $K_{n+1}$ associahedra.  
There are $n$ such associahedra since there are $n$ such arcs.

\begin{figure}[h]
\includegraphics{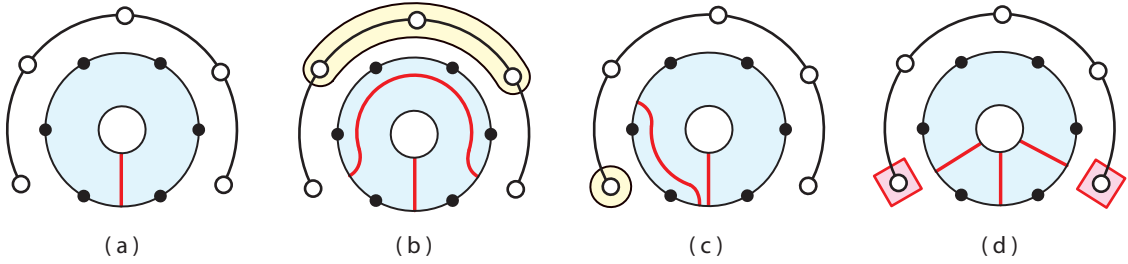}
\caption{Bijection between certain facets of halohedra and associahedra.}
\label{f:bijection-path}
\end{figure}

Finally, there exist arcs capturing $m$ marked boundary points.  Such arcs can be contracted and then normalized, as shown in Figure~\ref{f:contract-arc}.  This leads to a product structure of moduli spaces $\Mod_{(0,1)(0, m+1)}$ and $\Mod_{(0,2)(0, \la n-m+1, 0 \ra)}$, where the +1 in both markings appears from the contracted arc. 
Proposition~\ref{p:assoc} and Theorem~\ref{t:moduli-cube} show this facet to be $K_m \times \Y_{n-m+1}$.  
When $G$ is a cycle with $n$ nodes, we see from Corollary~\ref{c:enum} that the total number of facets of $\Y_n$  equals the number of facets of $\KG$ plus $n + 1$.  By Corollary~\ref{c:asscyc}, $\KG$ is the cyclohedron $W_n$, known to have $n^2-n$ facets.
\end{proof}

%%%%%%%%%%%%%%%%%%%%%%%%%%%%%%%%%%%%%%%%%%%%%%%%%%%%%%%%%%%%%%%%%%%%%%%%%%%%%%%%%%%%%
\subsection{}

From an algebraic perspective, the face poset of the associahedron $K_n$ is isomorphic to the poset of ways to associate $n$ objects on an interval.   Indeed, the associahedra characterize the structure of weakly associative products.  Classical examples of weakly associative product structures are the $A_n$ spaces, topological $H$-spaces with weakly associative multiplication of points. The notion of ``weakness'' should be understood as ``up to homotopy,'' where there is a path in the space from $(ab)c$ to $a(bc)$.  
The \emph{cyclic} version of this is the cyclohedron $W_n$, whose face poset is isomorphic to the ways to associate $n$ objects on a circle.  We establish such an algebraic structure for the halohedron $\Y_n$.   We begin by looking at the algebraic structure behind another classically known polytope, the multiplihedron.
 
The \emph{multiplihedra} polytopes, denoted $J_n$, were first discovered by Stasheff in \cite{jds2}.  They play a similar role for maps of loop spaces as associahedra play for loop spaces. 
The multiplihedron $J_n$ is a polytope of dimension $n-1$ whose vertices correspond to ways of associating $n$ objects \emph{and} applying an operation $f$.  At a high level, the multiplihedron is based on maps $f: A \to B$, where neither the range nor the domain is strictly associative.  The left side of Figure~\ref{f:multiplihedron} shows the 2D multiplihedron $J_3$ with its vertices labeled.  These polytopes have appeared in numerous areas related to deformation and category theories.  Notably, work by Forcey \cite{for} finally proved $J_n$ to be a polytope while giving it a geometric realization.  

\begin{figure}[h]
\includegraphics{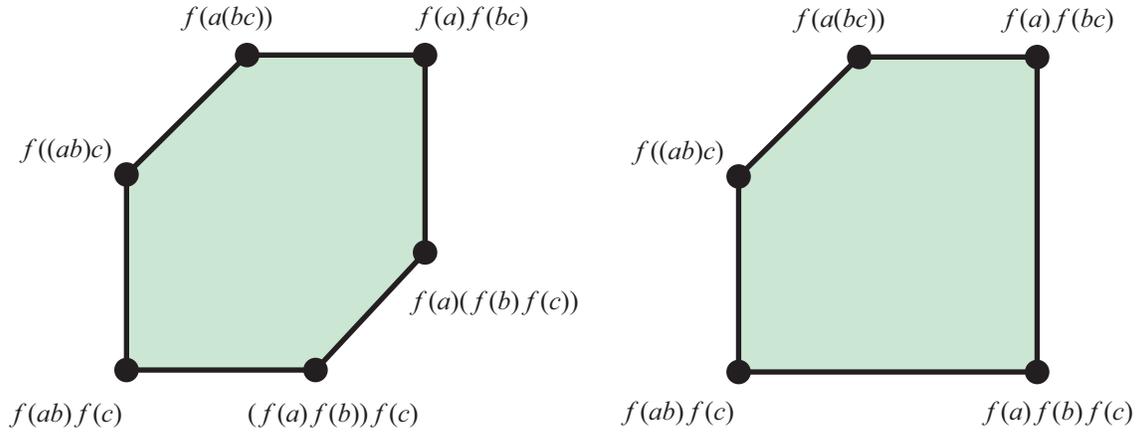}
\caption{The multiplihedron and the graph cubeahedron of a path.}
\label{f:multiplihedron}
\end{figure}

Recently, Mau and Woodward \cite{mw} were able to view the multiplihedra as a compactification of the moduli spaces of \emph{quilted} disks, interpreted from the perspective of cohomological field theories.  Here, a quilted disk is defined as a closed disk with a circle (``quilt'') tangent to a unique point in the boundary, along with certain properties.  

Halohedra and the more general graph cubeahedra fit into this larger context.  Figure~\ref{f:polytope-moduli2} serves as our guide.  Part (a) of the figure shows the moduli version of the halohedron.  Two of its facets are seen as (b) the cyclohedron and (c) the associahedron.  If the interior of the annulus is now colored black, as in  Figure~\ref{f:polytope-moduli2}(d), viewed not as a topological hole but a ``quilted circle,'' we obtain a cyclic version of the multiplihedron.  When an arc is drawn from the quilt to the boundary as in part (e), symbolizing (after contraction) a quilted circle tangent to the boundary, we obtain the quilted disk viewpoint of the multiplihedron of Mau and Woodward.

\begin{figure}[h]
\includegraphics[width=\textwidth]{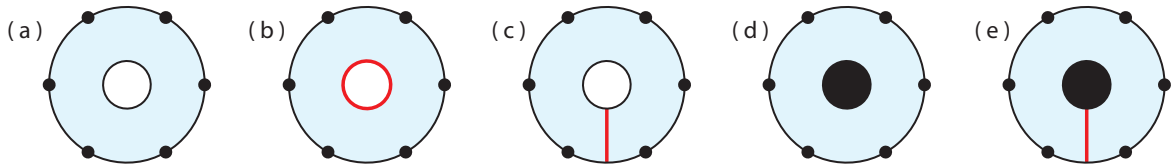}
\caption{The polytopes (a) halohedra, (b) cyclohedra, (c) associahedra, (d) cyclo-multiplihedra, (e) multiplihedra.}
\label{f:polytope-moduli2}
\end{figure}

The multiplihedron is based on maps $f: A \to B$, where neither the range nor the domain is strictly associative.  From this perspective, there are several important quotients of the multiplihedron, as given by the following table.
\begin{table}[h]
\noindent \renewcommand{\arraystretch}{1.6}{
\begin{tabular}{|l|c|c|c|}
\hline
\emph{Strictly Assocative} & \emph{$G$ path} & \emph{$G$ cycle } & \emph{general $G$} \\ \hline \hline
Both $A$ and $B$ & cube \cite{bv} & cube & cube \\
Only $A$ & composihedra \cite{for} & cycle composihedra & graph composihedra \\
Only $B$ & associahedra & halohedra & graph cubeahedra \\
Neither $A$ nor $B$ & \ multiplihedra \cite{jds2} \ & \ cycle multiplihedra \ & \ graph multiplihedra \cite{df} \ \\ \hline 
\end{tabular}}
\end{table}
The case of associativity of both range and domain is discussed by Boardman and Vogt \cite{bv}, where the result is shown to be the $n$-dimensional cube.  
%CITE EARLIER WORK OF SUGAWARA `HOMOTOPY MULTIPLICATIVE'
The case of an associative domain is described by Forcey \cite{for}, where the new quotient of the multiplihedron is called the \emph{composihedron}; these polytopes are the shapes of the axioms governing composition in higher enriched category theory.  The classical case of a strictly associative range (for $n$ letters in a path) was originally described in \cite{jds2}, where Stasheff shows that the multiplihedron $J_n$ becomes the associahedron $K_{n+1}$.  The right side of Figure~\ref{f:multiplihedron} shows the underlying algebraic structure of the 2D associahedron viewed as the graph cubeahedron of a path by Proposition~\ref{p:k-path}.
%Figure~\ref{f:algebra-halo} shows the underlying algebraic structure of the vertices of the graph cubeahedra $\CG$ when $G$ is a path.  By Proposition~\ref{p:k-path}, this is the Stasheff associahedron.

%\begin{figure}[h]
%\includegraphics{mod-3d-function.eps}
%\caption{Algebraic structure of the vertices of the associahedron $K_5$.}
%\label{f:algebra-halo}
%\end{figure}

\begin{figure}[h]
\includegraphics{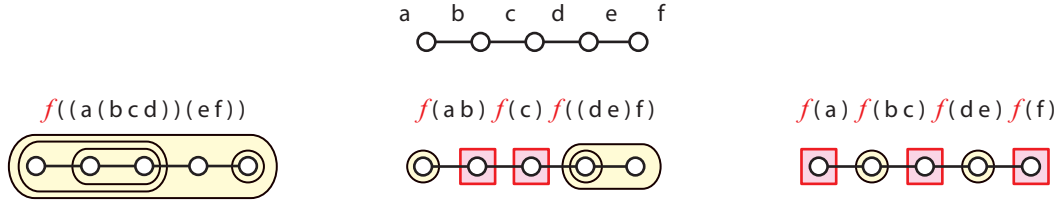}
\caption{Bijection between design tubes and associativity.}
\label{f:bi-path-multi}
\end{figure}

Figure~\ref{f:bi-path-multi} sketches a bijection between the associahedra (as design tubings of 5 nodes from Figure~\ref{f:bijection-path-assoc}) and associativity/function operations on 6 letters.  
Recall that the face poset of the cyclohedron $W_n$ is isomorphic to the ways to associate $n$ objects cyclically.  This is a natural generalization of $K_n$ where its face poset is recognized as the ways of associating $n$ objects linearly.  Instead, if we interpret $K_{n+1}$ in the context of Proposition~\ref{p:k-path}, the same natural cyclic generalization gives us the halohedron $\Y_n$.  In a broader context, the generalization of the multiplihedron to arbitrary graph is given in \cite{df}.  A geometric realization of these objects as well as their combinatorial interplay is provided there as well.

%%%%%%%%%%%%%%%%%%%%%%%%%%%%%%%%%%%%%%%%%%%%%%%%%%%%%%%%%%%%%%%%%%%%%%%%%%%%%%%%%%%%%
%
%                  References
%
%%%%%%%%%%%%%%%%%%%%%%%%%%%%%%%%%%%%%%%%%%%%%%%%%%%%%%%%%%%%%%%%%%%%%%%%%%%%%%%%%%%%%
\bibliographystyle{amsplain}

\end{document}